\documentclass[11pt]{amsart}
\usepackage[T1]{fontenc}
\usepackage{amssymb}
\usepackage{mathrsfs}
\usepackage{amsmath, amsfonts,enumerate}
\usepackage{graphics}
\usepackage{verbatim}
\usepackage{amsthm}
\usepackage{latexsym,bm}
\usepackage{euscript}
\usepackage{amscd}
\makeatletter \@mparswitchfalse \makeatother
\newtheorem{theorem}{Theorem}
\newtheorem{lemma}[theorem]{Lemma}

\newtheorem{proposition}[theorem]{Proposition}
\theoremstyle{remark}

\theoremstyle{definition}
\newtheorem{remark}[theorem]{Remark}
\newtheorem{definition}[theorem]{Definition}

\numberwithin{equation}{section}
\numberwithin{theorem}{section}

\def\F{\mathcal{F}}

\def\R{\mathbb R}

\let\<\langle
\def\ch{\raise 0.5ex \hbox{$\chi$}}

\def\E{\mathbb{E}}

\let\\\cr
\let\phi\varphi

\let\epsilon\varepsilon

\def\Re{\operatorname{Re}}

\def\log{\operatorname{log}}
\newcommand{\e}{\varepsilon}

\begin{document}
\title[Weighted inequalities]{The sharp weighted maximal inequalities \\ for noncommutative martingales}

\author[T. Ga\l \k{a}zka]{Tomasz Ga\l \k{a}zka}
\address{Faculty of Mathematics, Informatics and Mechanics, University of Warsaw,
Banacha 2, 02-097 Warsaw, Poland}
\email{T.Galazka@mimuw.edu.pl}

\author[Y. Jiao]{Yong Jiao}
\address{School of Mathematics and Statistics, Central South University, Changsha 410085,
People's Republic of China}
\email{jiaoyong@csu.edu.cn}

\author[A. Os\k ekowski]{Adam Os\k{e}kowski}
\address{Faculty of Mathematics, Informatics and Mechanics, University of Warsaw,
Banacha 2, 02-097 Warsaw, Poland}
\email{A.Osekowski@mimuw.edu.pl}

\author[L. Wu]{Lian Wu}
\address{School of Mathematics and Statistics, Central South University, Changsha 410085,
People's Republic of China}
\email{wulian@csu.edu.cn}

\subjclass[2010]{Primary: 46L53; Secondary: 60G42}
\keywords{Operator-valued martingales, weighted maximal inequalities, $A_p$ weights, singular integrals}
\thanks{Tomasz Ga\l\k{a}zka is supported by Narodowe Centrum Nauki (Poland), grant 2018/30/Q/ST1/00072. Yong Jiao is supported by the NSFC (No. 11961131003 and No. 11722114). Adam Os\k{e}kowski is supported by Narodowe Centrum Nauki (Poland), grant 2018/30/Q/ST1/00072. Lian Wu is supported by the NSFC (No.11971484)}

\begin{abstract}
The purpose of the paper is to establish weighted maximal $L_p$-inequalities in the context of operator-valued martingales on semifinite von Neumann algebras. The main emphasis is put on the optimal dependence of the $L_p$ constants on the characteristic of the weight involved. As applications, we establish weighted estimates for the noncommutative version of Hardy-Littlewood maximal operator and weighted bounds for noncommutative maximal truncations of a wide class of singular integrals.
\end{abstract}

\maketitle

\section{Introduction}
The theory of noncommutative martingales is a fast-expanding area of mathematics, and its fruitful connections with the theory of operator algebras and noncommutative harmonic analysis have been evidenced in numerous articles; see for instance \cite{pisierxu, PisierShlyakhtenko, Xu2006, R3, Mei, MP,parcet,JOW-continuous-dif-sub, JOWA,JOW2,JNWZ}. One of the primary goals of this paper is to study the context of maximal inequalities for operator-valued martingales in the presence of a weight, i.e., a nonnegative and integrable function.

To present the results from the appropriate perspective, let us discuss several closely related areas in the literature. For the relevant definitions and notations, we refer the reader to the next section. The fundamental results of Doob assert that if $x=(x_n)_{n\geq 0}$ is a martingale on some classical probability space $(\Omega,\F,\mathbb{P})$, then we have the weak-type estimate
$$ \lambda \mathbb{P}(\sup_{n\geq 0}|x_n|\geq \lambda)\leq \|x\|_{L_1},\qquad \lambda>0,$$
and its strong-type analogue
$$ \left\|\sup_{n\geq 0}|x_n|\right\|_{L_p}\leq \frac{p}{p-1}\|x\|_{L_p},\qquad 1<p\leq \infty.$$

One may ask about the noncommutative version of the above estimates. In this new context the martingale becomes a sequence of operators and one of the difficulties which need to be overcome is the lack of maximal functions.
In the celebrated paper \cite{cucu}, Cuculescu proposed the following approach towards the weak-type estimate. Suppose that $x=(x_n)_{n\geq 0}$ is an $L_1$-bounded martingale on a filtered, tracial von Neumann algebra $(\mathcal{M},\tau)$. Then for any $\lambda>0$ there is a projection $q_\lambda$ such that $q_\lambda x_nq_\lambda\leq \lambda$ for each $n$ and
\begin{displaymath}
\lambda\tau\left(I-q_\lambda\right) \leq \|x\|_{L_1(\mathcal{M})}.
\end{displaymath}
It is easy to see that this estimate does extend the above weak-type bound, the projection $I-q_\lambda$ plays the role of the indicator function of the event $\{\sup_{n\geq 0}|x_n|\geq \lambda\}$.

Thirty years later, the corresponding strong-type bound was proved in \cite{doob}, thus obtaining the noncommutative analogue of the classical result of Doob. The main idea is to directly introduce the maximal $L_p$-norm of a martingale directly, exploiting vector-valued $L_p$ spaces  $L_p\left(\mathcal{M};\ell_{\infty}\right)$ introduced by Pisier \cite{pisier} in the mid-nineties. The result can be formulated as
\begin{equation}\label{noncommdoob}
 \|x\|_{L_p(\mathcal{M};\ell_\infty)}\lesssim \left(\frac{p}{p-1}\right)^2 \|x\|_{L_p(\mathcal{M})},\qquad 1<p\leq \infty,
\end{equation}
where $\lesssim$ means that the inequality holds true up to some absolute constant and the quadratic order $O((p-1)^{-2})$ as $p\to 1$ is the best possible (see \cite{JX2}). In the proof of  the above inequality, the author transferred the problem to the dual estimate
\begin{equation}\label{dualdoob}
 \qquad \qquad \qquad \qquad  \left\|\sum_{n\geq 0} \mathcal{E}_na_n\right\|_{L_p(\mathcal{M})}\lesssim p^2 \left\|\sum_{n\geq 0}a_n\right\|_{L_p(\mathcal{M})},\qquad 1\leq p<\infty,
\end{equation}
and established it with the use of complex interpolation and Hilbert module theory arguments. A different proof, based on real interpolation, was given in \cite{jungexu}.

The motivation for the results obtained in this paper comes from a very natural question about the weighted analogue of \eqref{noncommdoob}. Let us recall some basic facts from the commutative setting, in which the theory of weighted estimates has been widely developed. Let $d\geq 1$ be a fixed dimension. The Hardy-Littlewood maximal operator $M$ on $\R^d$ acts on locally integrable functions $f:\R^d\to \R$ by the formula
\begin{displaymath}
Mf\left(x\right)=\sup \frac{1}{|Q|} \int_Q |f\left(y\right)|dy,
\end{displaymath}
where the supremum is taken over all cubes $Q$ containing $x$, having sides parallel to the axes. Let $w$ be a weight, i.e., a nonnegative and locally integrable function and let $1<p<\infty$ be a fixed exponent. In the seminal paper \cite{mucken}, Muckenhoupt characterized those $w$, for which the maximal operator is bounded as an operator on $L_p(w)$, i.e., those $w$, for which there exists a finite constant $C_{p,w}$ depending only on the parameters indicated such that
\begin{equation}\label{weightlp}
\int_{\mathbb{R}^d} \left(Mf\right)^pw\mbox{d}x \leq C_{p,w} \int_{\mathbb{R}^d} |f|^pw\mbox{d}x.
\end{equation}
He also studied the analogous problem for weak-type $(p,p)$ inequality:
\begin{equation}\label{weightweak}
\lambda^p\int_{\{ x: Mf\left(x\right)\geq \lambda\}} w\mbox{d}x \leq C_{p,w} \int_{\mathbb{R}^d} |f|^pw\mbox{d}x.
\end{equation}
It turns out that both inequalities are true if and only if $w$ satisfies the so-called $A_p$ condition. The latter means that the $A_p$ characteristic of $w$, given by
\begin{equation}\label{Ap}
[w]_{A_p} := \sup_Q \left(\frac{1}{|Q|} \int_Q w\mbox{d}x \right)\left(\frac{1}{|Q|}\int_Q w^{\frac{1}{1-p}}\mbox{d}x\right)^{p-1}
\end{equation}
(the supremum is taken over all cubes $Q\subset \R^d$ with sides parallel to the axes), is finite. Soon after the appearance of \cite{mucken}, it was shown that the $A_p$ condition characterizes the weighted $L_p$ and weak-$L_p$ boundedness of large families of classical operators, including the Hilbert transform (Hunt, Muckenhoupt and Wheeden \cite{HMW}), general Calde\'on-Zygmund singular integrals (Coifman and Fefferman \cite{CF}), fractional and Poisson integrals (Sawyer \cite{Sa1,Sa2}), area functionals (Buckley \cite{buckley},  Lerner \cite{Le1}) and many more. In  addition, following the work of Ikeda and Kazamaki \cite{IK} (see also Kazamaki \cite{K}), most of the results have been successfully transferred from the analytic to the probabilistic, martingale context (for some recent progress in this direction, see \cite{BO0,BO,BO1,O1,O2,O3}).

There is a very interesting aspect of the weight theory, concerning the extraction of the optimal dependence of the constants involved on the characteristic of a weight. Let us illustrate this problem on the estimate \eqref{weightlp} above. As we have already discussed above, if $w\in A_p$, then the inequality holds with some finite constant $C_{p,w}$. The question is: given $1<p<\infty$, what is the optimal (i.e., the least) exponent $\kappa_p$ such that $C_{p,w}\leq c_p[w]_{A_p}^{\kappa_p}$ for some constant $c_p$ depending only on $p$? This topic has appeared for the first time in Buckley's work \cite{buckley}, where it was shown that in the context of maximal functions, the exponent $\kappa_p=1/(p-1)$ is the best. The breakthrough about the $A_2$ conjecture for general Calderon-Zygmund operators was triggered by Hytonen \cite{Hyt}. For similar results for other classical operators, see e.g. \cite{BO,LMPT,Le1,O3} and consult the references therein.

There is a natural question how much of the weighted theory can be carried over to the noncommutative setting discussed previously. A partial answer to this question was provided in the context of matrix weights, which has been developed very intensively during the last decade. We will discuss here only the extension of Muckenhoupt's estimates. Suppose that $n>1$, $d\geq 1$ are fixed integers. A matrix weight $W$ is an $n\times n$ self-adjoint matrix function on $\R^d$ (with locally integrable entries) such that $W(x)$ is nonnegative-definite for almost all $x\in \R^d$. Given $1\leq p<\infty$ and an $n\times n$ matrix weight $W$ on   $\R^d$, we define the associated weighted space $L_p(W)$ to be the class of all measurable, vector-valued functions $f:\R^d\to \R^n$ such that
$$ \|f\|_{L_p(W)}=\left(\int_{\R^d}|W(x)^{1/p}f(x)|^p\mbox{d}x\right)^{1/p}<\infty.$$
One of the challenging problems (see also \eqref{noncommdoob} above) is to generalize efficiently the Hardy-Littlewood maximal operator $M$ to this new setting. Another question arising immediately concerns the appropriate interpretation of the boundedness of this operator on weighted spaces: since $M$ acts between different spaces (it is reasonable to expect $Mf$ to be a nonnegative function on $\R^d$), the symbol $\|Mf\|_{L_p(W)}$ simply makes no sense. To handle this difficulty, it is instructive to inspect the following change-of-measure argument. Namely, for a given linear operator $T$, its boundedness on the space $L_p(W)$ is equivalent to the boundedness of $W^{1/p}TW^{-1/p}$ on the (unweighted) space $L_p(\R^n;\R^d)$. This suggests that for the maximal operator, one should compose it appropriately with powers of the weight $W$, and then study the boundedness of the resulting operator on the usual unweighted spaces. This idea has turned out to be successful, and it has been generalized to the wide class of Calder\'on-Zygmund singular integral operators by a number of authors (cf. \cite{BPW,CIM,CPO,Go,NPTV}), as well as to the context of fractional operators \cite{CIM}. We mention that the first sharp $L_2$ estimate with weight matrix for the square dyadic function is due to Hyt\"onen et al. in \cite{HPV}.

However, essentially nothing is known in the context of martingales on tracial von Neumann algebras. While it is natural to treat the weights and martingales as operators, the noncommutativity makes the analysis of the joint behavior of these objects extremely difficult. We have decided to restrict ourselves to the special, semicomutative setting, in which many technicalities disappear, but on the other hand, the questions are still interesting and challenging. Namely, we will assume that the underlying von Neumann algebra is of the form $\mathcal{M}=L_\infty(X,\F,\mu)\overline\otimes \mathcal{N}$, where $(X,\F,\mu)$ is a classical $\sigma$-finite measure space and $\mathcal{N}$ is another von Neumann algebra. We will consider filtrations which act on the first component only (i.e., are of the form $L_\infty(X,\F_n,\mu)\overline \otimes \mathcal{N}$, $n=0,\,1,\,\,2,\,\ldots$). Furthermore, a weight will be a nonnegative operator of the form $w\otimes I$, and hence it will commute with every element of $\mathcal{M}$: this, in particular, allows a very simple (and natural) definition of the $A_p$ characteristic: $[w\otimes I]_{A_p}:=[w]_{A_p}$. With no risk of confusion, we will often simplify the notation and write $w$ instead of $w\otimes I$.

Note that $\mathcal{M}$ can be identified with $L_\infty(X;\mathcal{N})$, the space of functions on $X$ taking values in $\mathcal{N}$. That is,  we will consider the case of operator-valued martingales on $X$, with weights being elements of the commutant. This semicomutative context has been studied by many authors and applied in various problems of noncommutative harmonic analysis; we mention here the excellent exposition  \cite{Mei} by Mei.

We will establish the following statement.

\begin{theorem}\label{ref}
Let $1<p<\infty$. Then for any $x\in \mathcal{M}$ and any weight $ {w}\in A_p$,
\begin{equation}\label{wdoob}
 \|x\|_{L_p^{ {w}}(\mathcal{M};\ell_\infty)}\lesssim_p [{w}]_{A_p}^{1/(p-1)}\|x\|_{L_p^ {w}(\mathcal{M})}.
\end{equation}
The exponent $1/(p-1)$ is the best possible, since it is already optimal in the classical case.
\end{theorem}

An important comment about the commutative setting is in order. The classical version of the above statement was first obtained by Buckley \cite{buckley} in the special dyadic case, with the use of interpolation and self-improving properties of Muckenhoupt's weights. An alternative proof in the commutative setting, basing on Bellman function method, can be found in \cite{O34}, but it still exploits some regularity of martingales. Indeed, the continuity of paths is used there. Without any additional regularity assumptions, the classical version of Theorem \ref{ref} for martingales adapted to general  filtration was established in \cite{lerner, O1}.  We point out that the proofs provided in \cite{lerner, O1} rely on the development of new ideas:  change-of-measure arguments,  sparse operators and the Bellman function method. The traditional techniques which are frequently used in the theory of weights (e.g., self-improvement, reverse H\"older inequalities) simply fail to hold for the general filtration (see Remark \ref{nonhomog}).

In our considerations below, we will study the estimate \eqref{wdoob} without any assumption on the filtration. Both `commutative' proofs, presented in \cite{lerner} and \cite{O1}, exploit a number of pointwise estimates which are no longer valid in the context of operators. %Fortunately, some special estimates and a certain change-of-measure argument can still be used.
The proof of Theorem \ref{ref} will be divided into two cases according to the exponent $p$. For $1<p\leq 2$, we adopt a duality approach according to the sharp weighted dual Doob's inequality. However, for $2<p<\infty$, we have to invent a completely different method.

The paper is organized as follows. Some preliminary results and notation are presented in Section $2$. In Section $3$, we prove the main theorem, which is noncommutative weighted Doob's inequality with optimal dependence on the characteristic $[w]_{A_p}$. A corresponding weak weighted $L_p$-bound with optimal dependence on $[w]_{A_p}$ is also provided in this section. Sections $4$ and $5$ contain applications of Theorem \ref{ref}. In Section 4, we study a weighted version of the noncommutative $L_p$ bound in the context of maximal operators on general metric spaces satisfying the doubling condition. These results, which can be regarded as noncommutative version of \eqref{weightlp} and \eqref{weightweak}, extend Mei's result \cite[Chapter 3]{Mei} to the weighted case.  Section 5 is devoted to a weighted noncommutative $L_p$ bound for maximal truncations of a certain wide class of singular integral operators on $\R$. In the Appendix, the final part of the paper, we have decided to present alternative proofs of \eqref{wdoob}. Although the arguments yield suboptimal dependence on the characteristic (and exploit stronger regularity assumptions on the filtration), we believe that they are of independent interest and connections.

\section{Preliminaries}

We will introduce and discuss here some basic facts from the operator theory which will be needed in our later considerations. For the detailed and systematic exposition of the subject we refer the reader to the monographs \cite{operator2} and \cite{operator1}.

\subsection{Measurable operators}
Throughout, the letter $\mathcal{M}$ will stand for a semifinite von Neumann algebra of operators acting on some given Hilbert space $\mathcal{H}$, equipped with a  faithful and normal trace $\tau$. Let $x$ be a densely defined self-adjoint operator on $\mathcal{H}$, with the spectral resolution
\begin{displaymath}
x  = \int_{-\infty}^{\infty} sde^x_s.
\end{displaymath}
Then for any Borel set  $B \subset \mathbb{R}$, we define the associated spectral projection by
\begin{displaymath}
{I}_B\left(x\right) = \int_{-\infty}^{\infty} \chi_B\left(s\right)de^x_s.
\end{displaymath}
One can similarly introduces the operator $f(x)$ for sufficiently regular function $f$ on $\R$.
A closed, densely defined operator $x$ on $H$ is said to be  {affiliated} with $\mathcal{M}$ if for all unitary operators $u$ belonging to the commutant $\mathcal{M}'$ of $\mathcal{M}$, we have the identity $u^*au=a$. An operator $x$ affiliated with $\mathcal{M}$ is said to be $\tau$-\textit{measurable}, if there is  $s\geq 0$ such that $\tau\left(I_{\left(s,\infty\right)}\left(|x|\right)\right) < \infty$, where $|x| = \left(x^*x\right)^{1/2}$. We denote the space of all $\tau$-measurable operators by $L_0\left(\mathcal{M},\tau\right)$. Then, for all $0 < p \leq \infty$, the noncommutative $L_p$ space associated with $\left(\mathcal{M},\tau\right)$  is $$L_p\left(\mathcal{M},\tau\right) = \lbrace x\in L_0\left(\mathcal{M},\tau\right) : \tau\left(|x|^p\right) < \infty \rbrace.$$
The associated (semi-)norm is defined by $\displaystyle \|x\|_p = \left(\tau\left(|x|^p\right)\right)^{1/p},$
which is understood as the standard operator norm in the boundary case $p=\infty$.

\subsection{Noncommutative martingales and martingale transforms}
A filtration is an increasing sequence $\left(\mathcal{M}_n\right)_{n \geq 0}$  of von Neumann subalgebras of $\mathcal{M}$ such that the union $\bigcup_{n \geq 0} \mathcal{M}_n$ is w$^\ast$-dense in $\mathcal{M}$. In such a case, for each $n \geq 0$ there is a conditional expectation $\mathcal{E}_n:\mathcal{M} \to \mathcal{M}_n$ associated with $\mathcal{M}_n$: one defines this object as the dual map of natural inclusion
 $i : L_1\left(\mathcal{M}_n\right) \to L_1\left(\mathcal{M}\right)$. It can be easily verified that  $ \mathcal{E}_n\left(axb\right) = a\mathcal{E}_n\left(x\right)b$  for all $x\in \mathcal{M}$ and $a,\,b\in \mathcal{M}_n $; furthermore, $\mathcal{E}_n$ is $\tau$-preserving, i.e., we have $ \tau \circ \mathcal{E}_n = \tau$. In addition, the collection of conditional expectations satisfies the tower property $\mathcal{E}_n\mathcal{E}_m = \mathcal{E}_m\mathcal{E}_n = \mathcal{E}_{\min{\left(m,n\right)}} $. Finally, for any $1\leq p\leq \infty$, the operator  $\mathcal{E}_n$ extends to a contractive projection from $L_p\left(\mathcal{M},\tau\right)$ onto $L_p\left(\mathcal{M}_n,\tau_{|\mathcal{M}_n}\right)$.

A sequence $x = \left(x_n\right)_{n \geq 0} \subset L_1\left(\mathcal{M}\right)+L_\infty(\mathcal{M})$ is called a martingale with respect to  $\left(\mathcal{M}_n\right)_{n \geq 0}$, if the equality $\mathcal{E}_{n}\left(x_{n+1}\right) = x_{n}$ holds for all  $n \geq 0$. If, in addition, $x_n \in L_p\left(\mathcal{M}\right)$ for all $n\geq 0$, then $x$ is called an $L_p$-martingale with respect to  $\left(\mathcal{M}_n\right)_{n \geq 0}$ and we set
\begin{displaymath}
\|x\|_p = \sup_{n \geq 0}\|x_n\|_p.
\end{displaymath}
The martingale $x$ is said to be $L_p$-bounded if $\|x\|_p<\infty$.

Given a martingale $x=(x_n)_{n\geq 0}$, we define its difference sequence $dx=(dx_n)_{n\geq 0}$ by $dx_0=x_0$ and $dx_n=x_n-x_{n-1}$, $n=1,\,2,\,\ldots$. A martingale $y=(y_n)_{n\geq 0}$ is called a transform of $x=(x_n)_{n\geq 0}$, if there is a deterministic sequence $\e=(\e_n)_{n\geq 0}$ with values in $[-1,1]$ such that  $dy_n=\e_ndx_n$ for all $n\geq 0$. Martingale transforms satisfy the $L_p$ estimate
\begin{equation}\label{Lp}
 \|y_n\|_p\leq C_p\|x_n\|_p,\qquad n=0,\,1,\,2\,\ldots,\,\, 1<p<\infty,
\end{equation}
for some constant $C_p$ depending only on $p$. Actually, it can be shown that the optimal orders, as $p\to 1$ or $p\to \infty$, are $O((p-1)^{-1})$ and $O(p)$, respectively. Furthermore, one can allow a slightly larger class of transforming sequences $\e$. See \cite{pisierxu} and \cite{rand} for more on this subject.

\smallskip

\subsection{Maximal spaces} Now we discuss the suitable space required to define meaningful maximal functions. Let $1\leq p\leq \infty$. We define $L_p\left(\mathcal{M}; \ell_\infty\right)$ as the space of all sequences $x = \left(x_n\right)_{n \geq 0} \subset L_p\left(\mathcal{M}\right)$ which admit the decomposition
\begin{displaymath}
x_n = ay_n b \ \ \ \ \ \ \ \ \hbox{for all} \ n \geq 0,
\end{displaymath}
for some  $a,\,b \in L_{2p}\left(\mathcal{M}\right)$ and $y = \left(y_n\right)_{n \geq 0} \subset L_\infty\left(\mathcal{M}\right)$. We equip this space  with the norm
\begin{displaymath}
\|x\|_{L_p\left(\mathcal{M}; \ell_\infty\right)} = \inf \left\lbrace\|a\|_{2p} \sup_{n \geq 0} \|y_n\|_{\infty} \|b\|_{2p}\right\rbrace,
\end{displaymath}
where infimum runs over all factorizations of $x$ as above. The following dual reformulation will be important to us later. Namely, we define $L_p\left(\mathcal{M}; \ell_1\right)$ as the space of all sequences $x = \left(x_n\right)_{n \geq 0} \subset L_p\left(\mathcal{M}\right)$, which are of the form
\begin{displaymath}
x_n = \sum_{k \geq 0} u_{kn}^\ast v_{kn} \ \ \ \ \ \ \ \hbox{for all} \ n \geq 0,
\end{displaymath}
where families $\left(u_{kn}\right)_{k,n \geq 0}, \left(v_{kn}\right)_{k,n \geq 0} \subset L_{2p}\left(\mathcal{M}\right)$ satisfy
\begin{displaymath}
\sum_{k,n \geq 0} u_{kn}^\ast u_{kn} \in L_p\left(\mathcal{M}\right) \ \ \ \ \ \hbox{and} \ \ \ \ \ \sum_{k,n \geq 0} v_{kn}^\ast v_{kn} \in L_p\left(\mathcal{M}\right).
\end{displaymath}
The space $L_p\left(\mathcal{M}; \ell_1\right)$ is equipped with the norm
\begin{displaymath}
\|x\|_{L_p\left(\mathcal{M}; \ell_1\right)} = \inf \Bigg\lbrace\left\|\sum_{k,n \geq 0} u_{kn}^\ast u_{kn}\right\|_p^{\frac{1}{2}} \left\|\sum_{k,n \geq 0} v_{kn}^\ast v_{kn}\right\|_p^{\frac{1}{2}}\Bigg\rbrace,
\end{displaymath}
where infimum runs over all decompositions of $x$ as above. Both  $L_p\left(\mathcal{M}; \ell_\infty\right)$ and $L_p\left(\mathcal{M}; \ell_1\right)$ are Banach spaces and the following theorem is true (see \cite{doob}).

\begin{theorem}
\label{dual}
Let $1 \leq p < \infty$ and $p'$ be the conjugate of $p$. Then
\begin{displaymath}
L_{p}\left(\mathcal{M}; \ell_1\right)^*=L_{p'}\left(\mathcal{M}; \ell_\infty\right) \ \ \ \ \ \hbox{isometrically}
\end{displaymath}
with the duality bracket given by
\begin{displaymath}
\left(x,y\right) = \sum_{n \geq 0} \tau\left(x_ny_n\right)
\end{displaymath}
for $x \in L_p\left(\mathcal{M}; \ell_\infty\right)$ and $y \in L_{p'}\left(\mathcal{M}; \ell_1\right)$.
\end{theorem}

The above spaces have  a much simpler description when restricted to nonnegative operators. Consider $x = \left(x_n\right)_{n \geq 0}$, where $x_n \geq 0$ for all $n \geq0$. Then we have
$$
\|x\|_{L_{p}\left(\mathcal{M}; \ell_1\right)} = \Big\|\sum_{n \geq 0} x_n\Big\|_{L_p}.$$ Furthermore,  $x$ belongs to $L_p\left(\mathcal{M}; \ell_\infty\right)$ if and only if there exists a positive operator $a \in L_p\left(\mathcal{M}\right)$ such that $x_n \leq a$ for all $n \geq 0$. In addition, $\|x\|_{ L_p\left(\mathcal{M}; \ell_\infty\right)} = \inf\lbrace \|a\|_{L_p} : x_n \leq a\mbox{ for all }n\rbrace$.

We would also like to conclude with the remark that the definition of $L_p(\mathcal{M};\ell_\infty)$ extends easily to the case in which the sequences are indexed by an arbitrary set $I$: the relevant factorization makes perfect sense. Denoting the corresponding space by $L_p(\mathcal{M};\ell_\infty(I))$, it is not difficult to check the identity
\begin{equation}\label{not_difficult}
 \|x\|_{L_p(\mathcal{M};\ell_\infty(I))}=\sup_{J\text{ finite}}\|x\|_{L_p(\mathcal{M};\ell_\infty(J))}.
\end{equation}
This observation, with $I=\mathbb{Z}$ or $I=[0,\infty)$, will be important for our applications below.

\smallskip

\subsection{Martingale weights.}
In this subsection, we introduce some basic information on weighted theory in the commutative context. Suppose that $(X,\F,\mu)$ is a classical measure space, filtered by $(\F_n)_{n\geq 0}$, a nondecreasing family of sub-$\sigma$-fields of $\F$ such that $\sigma\left(\bigcup_{n\geq 0}\F_n\right)=\F$ and such that $(X,\F_0,\mu)$ is $\sigma$-finite.
A weight is a positive function belonging to $L_1(X)+L_\infty(X)$; typically, such an object will be denoted by $u$, $w$ or $v$. Any weight $w$ gives rise to the corresponding measure on $X$, also denoted by $w$, and defined by $w(A)=\int_A w\mbox{d}\mu$ for all $A\in \mathcal{F}$. Given $1<p<\infty$, a weight $w$ satisfies the (martingale) $A_p$ condition, if the $A_p$ characteristic
$$ [w]_{A_p}=\sup_{n\geq 0}\left\|\mathbb{E}_n (w) \mathbb{E}_n(w^{1/(1-p)})^{p-1}\right\|_{L_\infty(X)}$$
is finite.
If the filtration is atomic, i.e., for each $n$ the $\sigma$-field $\F_n$ is generated by pairwise disjoint sets of positive and finite measure, then the characteristic can be rewritten in the more usual form
$$ [w]_{A_p}=\sup \left(\frac{1}{\mu(Q)}\int_Q w\mbox{d}\mu\right)\left(\frac{1}{\mu(Q)}\int_Q w^{1/(1-p)}\mbox{d}\mu\right)^{p-1},$$
where the supremum is taken over atoms $Q$ of the filtration. The dual weight to $w\in A_p$ is given by $v=w^{1/(1-p)}$. It follows directly from the definition of the characteristic that $v\in A_{p'}$ and $[v]_{A_{p'}}=[w]_{A_p}^{1/(p-1)}$. There are versions of the $A_p$ condition in the boundary cases $p\in \{1,\infty\}$, which can be obtained by a simple passage to the limit. We will only present here the case $p=1$, as the choice $p=\infty$ will not be presented in our considerations. Namely, a weight $w$ satisfies Muckenhoupt's condition $A_1$, if its characteristic
$$ [w]_{A_1}=\sup_{n\geq 0} \left\|\mathcal{E}_n(w)/w\right\|_{L_\infty(X)}$$
is finite. If the filtration is atomic, then we have the identity
$$ [w]_{A_1}=\sup_Q \operatorname*{esssup}_X \frac{\frac{1}{\mu(Q)}\int_Q w\mbox{d}\mu}{w}.$$

\subsection{Noncommutative weighted $L_p$ spaces}
Assume that  $(X,\F,\mu)$ is a  classical measure space  and let $(\F_n)_{n\geq 0}$ be a discrete-time filtration such that $\sigma(\bigcup_{n\geq 0}\F_n)=\F$ and such that $(X,\F_0,\mu)$ is $\sigma$-finite. Suppose further that $\mathcal{N}$ is a given semifinite von Neumann algebra with a faithful, normal trace $\nu$. We set $\mathcal{M}=L_\infty(X,\F,\mu)\overline\otimes \mathcal{N}$ and endow this algebra with a standard tensor trace $\tau=\mu\otimes \nu$ and the filtration $\mathcal{M}_n=L_\infty(X,\F_n,\mu)\overline\otimes \mathcal{N}$, $n=0,\,1,\,2,\,\ldots$. Then the associated conditional expectations are given by $\mathcal{E}_n=\E(\cdot|\F_n)\otimes I_\mathcal{N}$, where $\E(\cdot|\F_n)$ is the classical conditional expectation with respect to $\mathcal{F}_n$. Furthermore, the elements of $\mathcal{M}$ can be regarded as bounded functions taking values in $\mathcal{N}$ and the $L_p$-bounded martingales in this context can be identified with $L_p$-bounded martingales on $(X,\F,\mu)$ with values in $L_p(\mathcal{N})$.

In our considerations below, a weight will be a positive operator of the form $w\otimes I$, where $w$ is a classical weight on $(X,\F,\mu)$. Such operators commute with all elements of $\mathcal{M}$ and all conditional expectations $\mathcal{E}_n(w\otimes I)$ also enjoy this property. We say that $ { {w\otimes I}}$ satisfies Muckenhoupt's condition $A_p$ (or belongs to the $A_p$ class), if the scalar weight $w$ has this property. Furthermore, we set $[ { {w}\otimes I}]_{A_p}=[w]_{A_p}$. From now on, we will skip the tensor and identify $w\otimes I$ with $w$; this should not lead to any confusion.

Given $1\leq p<\infty$ and $w$ as above, the associated noncommutative weighted $L_p$ space is defined by
$$ L_p^ { {w}}(\mathcal{M})=\left\{x\in L_0(\mathcal{M},\tau)\,:\,x { {w}}^{1/p}\in L_p(\mathcal{M})\right\}.$$
That is to say, $L_p^ { {w}}(\mathcal{M})$ is the usual noncommutative $L_p$ space with respect to the weighted trace $\tau^{ { {w}}}(x):=\tau(x { {w}})$, $x\in \mathcal{M}$. The fact that $ { {w}}$ is positive and commutes with all the elements of $\mathcal{M}$ implies that $\tau^{ { {w}}}$ is indeed a trace. This change-of-measure argument, based on passing from one trace to another, will play an important role in our considerations below. In particular, we will need the following simple fact. Here and in what follows, $\mathcal{E}_n^{ { {w}}}$ denotes the conditional expectation, with respect to $\mathcal{M}_n$ and the trace $\tau^{ { {w}}}$ (while $\mathcal{E}_n$ is the usual conditional expectation, relative to the unweighted trace $\tau$).

\begin{lemma}
For any $x\in L_1^w(\mathcal{M})$ we have
\begin{equation}
\label{1}
\mathcal{E}^{ {{w}}}_n\left(x\right) = \mathcal{E}_n\left(x {w}\right)\left(\mathcal{E}_n\left( {w}\right)\right)^{-1}.
\end{equation}
\end{lemma}
\begin{proof}
Let us check whether the right-hand side of \eqref{1} enjoys all the properties of conditional expectation. Obviously, it belongs to $\mathcal{M}_n$. Furthermore, if $a,\,b$ are arbitrary elements of $\mathcal{M}_n$, then by the commuting property of $ { {w}}$,
$$\mathcal{E}_n\left(axb { {w}}\right)\left(\mathcal{E}_n\left( { {w}}\right)\right)^{-1}=
\mathcal{E}_n\left(ax { {w}}b\right)\left(\mathcal{E}_n\left( { {w}}\right)\right)^{-1}=
a\mathcal{E}_n\left(x { {w}}\right)b\left(\mathcal{E}_n\left( { {w}}\right)\right)^{-1}
=a\mathcal{E}_n\left(x { {w}}\right)\left(\mathcal{E}_n\left( { {w}}\right)\right)^{-1}b.$$
Finally, the right-hand side of \eqref{1} preserves the trace $\tau^{ { {w}}}$: indeed,
\begin{align*}
 \tau^{ { {w}}}\left(\mathcal{E}_n\left(x {  w}\right)\left(\mathcal{E}_n\left( {w}\right)\right)^{-1}\right)
&=\tau\left(\mathcal{E}_n\left(x { {w}}\right)\left(\mathcal{E}_n\left( { {w}}\right)\right)^{-1} { {w}}\right)\\
&=\tau\left(\mathcal{E}_n\left(x { {w}}\right)\left(\mathcal{E}_n\left( { {w}}\right)\right)^{-1}\mathcal{E}_n( { {w}})\right)=\tau\left(\mathcal{E}_n\left(x { {w}}\right)\right)=\tau(x { {w}})=\tau^{ { {w}}}(x).
\end{align*}
This proves the claim.
\end{proof}

\begin{remark}
The above lemma has a very transparent meaning if the underlying filtration $(\F_n)_{n\geq 0}$ is atomic. In such a case, we have the following explicit formula for $\mathcal{E}_n^ { {w}}$: if we identify $\mathcal{M}$ with operator-valued random variables, then
\begin{displaymath}
\mathcal{E}_n^ { {w}}x = \sum_{Q \in {At}_n} \frac{1}{w\left(Q\right)}\int_{Q} x\left(\omega\right) w\mu(\mbox{d}\omega) \cdot \raisebox{2pt}{$\chi$}_Q,
\end{displaymath}
where $At_n$ is the collection of all atoms of $\F_n$ and  $w\left(Q\right) = \int_Q w\mbox{d}\mu$.
\end{remark}

By a similar argument, which rests on the passage from the trace $\tau$ to its weighted version $\tau^w$, one defines the appropriate weighted maximal spaces $L_p^w(\mathcal{M};\ell_\infty)$ and $L_p^w(\mathcal{M};\ell_1)$.

\bigskip

\section{Weighted Doob's maximal inequality}

In this section, we provide the proof of  Theorem \ref{ref} and also establish its weak version. Both results are of optimal dependence on the characteristic of the weight involved.

\subsection{Weighted maximal inequalities}
The  purpose of this subsection is to prove Theorem \ref{ref}. We start with the following statement, which extends dual Doob's inequality to the weighted case.

\begin{theorem}
\label{doob2}
Let $1\leq  p <\infty$ and $ {w} \in A_p$. For any sequence $\left(a_n\right)_{n\geq 0}$ of positive elements of $L_p^ {w}(\mathcal{M})$ we have
\begin{equation}\label{idual}
\left\|\sum_{n\geq 0} \mathcal{E}_n\left(a_n\right)\right\|_{L_p^ {w}(\mathcal{M})} \leq c_p[ {w}\,]_{A_p}\left\|\sum_{n \geq 0}a_n\right\|_{L_p^ {w}(\mathcal{M})},
\end{equation}
where $c_p$ depends only on $p$. The exponent of $[ {w}\,]_{A_p}$ is the best possible.
\end{theorem}
\begin{proof}[Proof of Theorem \ref{doob2} for $p=1$ and $p\geq 2$] By the $\sigma$-finiteness of $(X,\F_0,\mu)$, a simple splitting argument allows us to assume that  $\mu(X)<\infty$. Then in particular $w$ must be integrable. Now, if $p=1$, then we have
$$ \left\|\sum_{n\geq 0} \mathcal{E}_n\left(a_n\right)\right\|_{L_1^ {  w}(\mathcal{M})}=\tau\left(\sum_{n\geq 0} \mathcal{E}_n\left(a_n\right) { {w}}\right)=\tau\left(\sum_{n\geq 0} a_n\mathcal{E}_n\left( { {w}}\right)\right)\leq [ { {w}}]_{A_1}\tau\left(\sum_{n\geq 0} a_n  { {w}}\right),$$
so the desired bound holds with $c_1=1$. Next, suppose that $p\geq 2$ and let $v=w^{1/(1-p)}$ be the dual weight  to $w$. Muckenhoupt's condition implies that for any $n\geq 0$,
\begin{equation}\label{App}
 \mathcal{E}_n( { {w}})\mathcal{E}_n( { {v}})=\E(w|\F_n)\E(v|\F_n)\otimes I\leq [w]_{A_p}(\E(v|\F_n))^{2-p}\otimes I\leq [w]_{A_p}\E(v^{2-p}|\F_n)\otimes I,
\end{equation}
where the last passage is due to Jensen's inequality (and the assumption $p\geq 2$). Now, fix a positive operator $g\in \mathcal{M}$ satisfying $\|g\|_{L_{p'}^ { {v}}} \leq 1$. Note that $g\in L_1(\mathcal{M})$: by H\"older's inequality, $\|g\|_{L_1(\mathcal{M})}\leq \|g\|_{L_{p'}^ {  v}}\|w\|_{L_1}^{1/p}<\infty$. By properties of conditional expectations, we may write
\begin{displaymath}
\tau\left(\sum_{n\geq 0}\mathcal{E}_n(a_n)g\right)=\sum_{n \geq 0} \tau\left(\mathcal{E}_n\left(a_n\right)g\right) = \sum_{n \geq 0} \tau\left(\mathcal{E}_n\left(a_n\right)\mathcal{E}_n\left(g\right)\right).
\end{displaymath}
Using the identity $\left(\ref{1}\right)$ and the commuting properties of $ { {w}}$, $ { {v}}$ and their conditional expectations, we obtain
\begin{align*}
\sum_{n \geq 0} \tau\left(\mathcal{E}_n\left(a_n\right)\mathcal{E}_n\left(g\right)\right) &= \sum_{n \geq 0} \tau\left(\mathcal{E}^{ { {v}}}_n\left(a_n { {v}}^{-1}\right)\mathcal{E}^{ { {w}}}_n\left(g { {w}}^{-1}\right)\mathcal{E}_n\left( { {v}}\right)\mathcal{E}_n\left( { {w}}\right)\right),
\end{align*}
which, by \eqref{App}, does not exceed
\begin{align*}
\sum_{n \geq 0} \tau\left(\mathcal{E}^{ { {v}}}_n\left(a_n { {v}}^{-1}\right)\mathcal{E}^{ { {w}}}_n\left(g { {w}}^{-1}\right)[ { {w}}]_{A_p}\mathcal{E}_n\left( { {v}}^{2-p}\right)\right)
&= [ { {w}}]_{A_p} \sum_{n \geq 0} \tau\left(\mathcal{E}^{ { {v}}}_n\left(a_n { {v}}^{-1}\right)\mathcal{E}^{ { {w}}}_n\left(g { {w}}^{-1}\right) { {v}}^{2-p}\right)\\
&= [ { {w}}]_{A_p} \sum_{n \geq 0} \tau\left(\mathcal{E}^{ { {v}}}_n\left(a_n { {v}}^{-1}\right)\mathcal{E}^{ { {w}}}_n\left(g { {w}}^{-1}\right) { {v}}^{\frac{1}{p}} { {w}}^{\frac{1}{p'}}\right).
\end{align*}
As we mentioned above, $g$ belongs to the space $L_1(\mathcal{M})$ and hence $g {  w}^{-1}\in L_1^ {  w}(\mathcal{M})$.
By noncommutative Doob's inequality in $L_{p'}$, applied to the nonnegative martingale $\left(\mathcal{E}^{ { {w}}}_n\left(g { {w}}^{-1}\right)\right)_{n \geq 0}$ (on von Neumann algebra $\left(\mathcal{M}, \tau^ { {w}}\right)$), there exists an operator $a$ such that $\mathcal{E}^{ { {w}}}_n\left(g { {w}}^{-1}\right)\leq a$ for every $n \geq 0$ and
\begin{displaymath}
\|a\|_{L_{p'}^ { {w}}(\mathcal{M})} \leq C_{p'}\|g { {w}}^{-1}\|_{L_{p'}^ { {w}}(\mathcal{M})} = C_{p'}\left(\tau\left(g^{p'} { {w}}^{-p'} { {w}}\right)\right)^{\frac{1}{p'}} =  C_{p'}\|g\|_{L_{p'}^ { {v}}} \leq C_{p'}.
\end{displaymath}
Here the last estimate follows from the assumption $\|g\|_{L_{p'}^ { {v}}(\mathcal{M})}\leq 1$ we imposed at the beginning. Consequently, by the tracial property (and the fact that $ { {w}}$ and $ { {v}}$ commute with all elements of $\mathcal{M}$) we get
$$  \tau\left(\mathcal{E}^{ { {v}}}_n\left(a_n { {v}}^{-1}\right)\mathcal{E}^{ { {w}}}_n\left(g { {w}}^{-1}\right) { {v}}^{\frac{1}{p}} { {w}}^{\frac{1}{p'}}\right)\leq
\tau\left(\mathcal{E}^{ { {v}}}_n\left(a_n { {v}}^{-1}\right)a { {v}}^{\frac{1}{p}} { {w}}^{\frac{1}{p'}}\right).$$
Therefore, by the H\"older inequality,
\begin{align*}
\tau\left(\sum_{n \geq 0} \mathcal{E}_n\left(a_n\right)g\right) &\leq [ { {w}}]_{A_p}  \tau\left(\sum_{n \geq 0}\left(\mathcal{E}^{ { {v}}}_n\left(a_n { {v}}^{-1}\right) { {v}}^{\frac{1}{p}}\right)a { {w}}^{\frac{1}{p'}}\right)\\
&\leq [ {  w}\,]_{A_p}  \left\|\sum_{n \geq 0}\mathcal{E}^{ { {v}}}_n\left(a_n { {v}}^{-1}\right)\right\|_{L_p^ { {v}}(\mathcal{M})}\|a\|_{L_{p'}^ { {w}}(\mathcal{M})}\\
&\leq C_{p'}C_p[ { {w}}]_{A_p} \left\|\sum_{n \geq 0}a_n { {v}}^{-1}\right\|_{L_p^ { {v}}(\mathcal{M})} = c_p[ { {w}}]_{A_p} \left\|\sum_{n \geq 0}a_n\right\|_{L_p^ { {w}}(\mathcal{M})}.
\end{align*}
Here in the last line we have exploited the dual form of Doob's inequality \eqref{dualdoob}, applied to the nonnegative sequence  $\left(a_n { {v}}^{-1}\right)_{n \geq 0}$ on von Neumann algebra $\left(\mathcal{M}, \tau^ { {v}}\right)$. To finish the proof, we specify $g=(\sum_{n\geq 0}\mathcal{E}_na_n)^{p-1} { {w}}/\|\sum_{n\geq 0}\mathcal{E}_na_n\|_{L_p^ { {w}}(\mathcal{M})}^{p-1}$: then $\|g\|_{L_{p'}^ { {v}}}=1$ and
$$ \tau\left(\sum_{n \geq 0} \mathcal{E}_n\left(a_n\right)g\right) =\left\|\sum_{n\geq 0} \mathcal{E}_n\left(a_n\right)\right\|_{L_p^ { {w}}(\mathcal{M})}. \qquad \qquad \qedhere$$
\end{proof}

\begin{remark}
The above reasoning can be repeated in the case $1<p<2$, but then \eqref{App} does not hold any more. Instead, we may write
\begin{align*}
 \mathcal{E}_n( { {w}})\mathcal{E}_n( { {v}})&=\E(w|\F_n)\E(v|\F_n)\otimes I\\
 &=\E(v^{1-p}|\F_n)\E(v|\F_n)\otimes I\\
 &=\E(v^{\frac{1}{1-p'}}|\F_n)^{p'-1}\E(v|\F_n)\E(v^{\frac{1}{1-p'}}|\F_n)^{2-p'}\otimes I\\
 &\leq [v]_{A_{p'}} \E(v^{1-p}|\F_n)^{2-p'} \otimes I\\
 &\leq [v]_{A_{p'}} \E(v^{(1-p)(2-p')}|\F_n) \otimes I \\
 & = [v]_{A_{p'}} \E(v^{2-p}|\F_n) \otimes I
\end{align*}
where the last inequality is due to Jensen's inequality and the assumption that $p<2$. Note that $[v]_{A_{p'}}=[w]_{A_p}^{1/(p-1)}$, so we get \eqref{idual}, but with the worse, nonlinear dependence $[ {  w}]_{A_p}^{1/(p-1)}$. To overcome this difficulty, we will use a different approach.
\end{remark}

The proof of Theorem \ref{doob2} in the range $1<p<2$ is postponed for a while. We now use the duality approach to prove Theorem \ref{ref} for $1<p\leq 2$.

\begin{proof}[Proof of Theorem \ref{ref} for $1<p\leq 2$] We deduce the assertion from the previous statement. Again, we may assume that $\mu(X)<\infty$.
Pick an arbitrary positive element $x$ of $L_p^ {  w}\left(\mathcal{M}\right)$. Then
$$\|x\|_{L_1(\mathcal{M})}\leq \|x\|_{L_p^ {  w}(\mathcal{M})}\| {  v}\|_{L_1(\mathcal{M})}^{1/p'}<\infty$$ and hence $(x_n)_{n\geq 0} = (\mathcal{E}_n\left(x\right))_{n\geq 0}$ is a well-defined $L_1$-bounded martingale on $(\mathcal{M},\tau)$. This sequence is contained in $L_p^ {  w}(\mathcal{M})$, by \eqref{idual}. Next, consider an arbitrary operator $y \in L_{p'}^ { {w}}\left(\mathcal{M}; \ell_1\right)$ and let $\left(a_{kn}\right)_{k,n \geq 0}, \left(b_{kn}\right)_{k,n \geq 0}$  be families of elements of $L_{2p'}^ {  w}\left(\mathcal{M}\right)$, satisfying
\begin{displaymath}
y_n = \sum_{k \geq 0} a_{kn}^\ast b_{kn} \ \ \ \ \ \ \ \hbox{for all} \ n \geq 0.
\end{displaymath}
Then, by H\"older's inequality and properties of conditional expectations,
\begin{align*}
\Bigg \vert\sum_{n \geq 0} \tau^ {  w} \left(x_ny_n\right)\Bigg\vert = \Bigg\vert\sum_{n,k \geq 0} \tau \left(\mathcal{E}_n\left(x\right)a_{kn}^\ast b_{kn} {  w}\right)\Bigg\vert &= \Bigg\vert\sum_{n,k \geq 0} \tau \left(\mathcal{E}_n\left(a_{kn}^\ast b_{kn} {  w}\right)x\right)\Bigg\vert\\
&= \Bigg\vert \tau \left(\sum_{n,k \geq 0} \left(\mathcal{E}_n\left(a_{kn}^\ast b_{kn} {  w}\right)  {  w}^{-\frac{1}{p}}\right)x {  w}^{\frac{1}{p}}\right)\Bigg\vert\\
&\leq \left\|xw^{\frac{1}{p}}\right\|_{L_{p}\left(\mathcal{M}\right)}\left\|{\sum_{n,k \geq 0} \mathcal{E}_n\left(a_{kn}^\ast b_{kn} {  w}\right) {  w}^{-\frac{1}{p}}}\right\|_{L_{p'}\left(\mathcal{M}\right)}\\
&=\|x\|_{L_{p}^ {  w}\left(\mathcal M\right)}\left\|{\sum_{n,k \geq 0} \mathcal{E}_n\left(a_{kn}^\ast b_{kn} {  w}\right)}\right\|_{L_{p'}^ {  v}\left(\mathcal{M}\right)}.
\end{align*}
Now, by  the H\"older inequality (\cite[Proposition 2.15]{doob}) and Theorem \ref{doob2} applied to $ {  v}\in A_{p'}$ (note that $p'\geq 2$), we may proceed as follows:
\begin{align*}
\left\|\sum_{n,k \geq 0} \mathcal{E}_n\left(a_{kn}^\ast  {  w}^{\frac{1}{2}} b_{kn} {  w}^{\frac{1}{2}}\right)\right\|_{L_{p'}^ {  v}\left(\mathcal M\right)} &\leq \left\|\sum_{n,k \geq 0} \mathcal{E}_n\left(a_{kn}^\ast a_{kn} {  w}\right)\right\|_{L_{p'}^ {  v}\left(\mathcal{M}\right)}^{\frac{1}{2}}\left\|\sum_{n,k \geq 0} \mathcal{E}_n\left(b_{kn}^\ast b_{kn} {  w}\right)\right\|_{L_{p'}^ {  v}\left(\mathcal{M}\right)}^{\frac{1}{2}}\\
&\leq c_{p'}[ {  v}]_{A_{p'}} \left\|\sum_{n,k \geq 0} a_{kn}^\ast a_{kn} {  w}\right\|_{L_{p'}^ {  v}\left(\mathcal M\right)}^{\frac{1}{2}}\left\|\sum_{n,k \geq 0} b_{kn}^\ast b_{kn} {  w}\right\|_{L_{p'}^ {  v}\left(\mathcal M\right)}^{\frac{1}{2}}\\
&= c_{p'}[ {  v}]_{A_{p'}} \left\|\sum_{n,k \geq 0} a_{kn}^\ast a_{kn}\right\|_{L_{p'}^ {  w}\left(\mathcal M\right)}^{\frac{1}{2}}\left\|\sum_{n,k \geq 0} b_{kn}^\ast b_{kn}\right\|_{L_{p'}^ {  w}\left(\mathcal M\right)}^{\frac{1}{2}}\\
&\leq c_{p'}[ {  v}]_{A_{p'}} \|y\|_{L_{p'}^ {  w}\left(\mathcal{M}; \ell_1\right)} \\
&= c_{p'}[ {  w}]_{A_{p}}^{1/(p-1)} \|y\|_{L_{p'}^ {  w}\left(\mathcal{M}; \ell_1\right)}.
\end{align*}
Since $\sum_{n \geq 0} \tau^ {  w} \left(x_ny_n\right)$ is the duality bracket between $L_p^ {  w}(\mathcal{M};\ell_\infty)$ and $L_{p'}^ {  w}(\mathcal{M};\ell_1)$, we obtain the desired estimate
$$ \|x\|_{L_p^ {  w}(\mathcal{M};\ell_\infty)}\leq c_p[ {  w}\,]_{A_p}^{1/(p-1)}\|x\|_{L_p^ {  w}(\mathcal{M})}$$
for positive $x$. The passage to general operators follows from a standard argument.
\end{proof}

Next, we turn to the case $2<p<\infty$. In this case, we have to invent a different method. Moreover, this argument does not work for $1\leq p<2.$

\begin{proof}[Proof of Theorem \ref{ref} for $p>2$] Fix $ {  w}\in A_p$. By standard decomposition, it is enough to show the claim for positive operators $x\in L_p^ { {w}}(\mathcal{M})$. Our goal is to majorize the martingale $(x_n)_{n\geq 0}$ by an operator, whose norm  in $L_p^ { {w}}(\mathcal{M})$ is not bigger than $[w]_{A_p}^{1/(p-1)}\|x\|_{L_p^ { {w}}(\mathcal{M})}$, up to some constant depending only on $p$.

 We begin with the observation that  $x^{p-1} { {v}}^{1-p}$ is positive and belongs to $L_{p'}^ { {v}}(\mathcal{M})$: this is due to the identity $\|x^{p-1} {  v}^{1-p}\|_{L_{p'}^ {  v}(\mathcal{M})}=\|x\|_{L_p^ {  w}(\mathcal{M})}^{p-1}$. Thus,  we may apply Doob's inequality in $L_{p'}^ {  v}(\mathcal M)$ to the nonnegative martingale  $\left(\mathcal{E}_n^{ {  v}}\left(x^{p-1} {  v}^{1-p}\right)\right)_{n \geq 0}$ in $(\mathcal{M},\tau^ {  v})$, obtaining an operator $a$ such that $ \mathcal{E}_n^{ {  v}}\left(x^{p-1} {  v}^{1-p}\right)\leq a$ for every $n \geq 0$ and
\begin{equation}\label{bounda1}
\|a\|_{L_{p'}^ {  v}(\mathcal{M})}\leq c_p\|x\|_{L_p^ {  w}(\mathcal{M})}^{p-1}.
\end{equation}
 Next, we apply Doob's inequality again, this time in $L_{p'}^{ { {w}}}(\mathcal{M})$, to the nonnegative martingale $\left(\mathcal{E}_n^{ {  w}}\left(a {  w}^{-1}\right)\right)_{n \geq 0}$. As the result, we get an operator $b$ such that $\mathcal{E}_n^{ {  w}}\left(a  { {w}}^{-1}\right)\leq b$ for $n \geq 0$ and whose norm satisfies
\begin{equation}\label{boundb1}
\|b\|_{L_{p'}^ {  w}(\mathcal{M})} \leq c_{p}\|a {  w}^{-1}\|_{L_{p'}^ {  w}\left(\mathcal M\right)}=c_p\|a\|_{L_{p'}^ { {v}}(\mathcal{M})}.
\end{equation}
Using the change of measure formula \eqref{1}, the fact that $ \mathcal{E}_n\left( {  w}\right)\left(\mathcal{E}_n\left( {  v}\right)\right)^{p-1}\leq [ {  w}\,]_{A_p}$ and the estimate $\mathcal{E}^{ {  v}}_n\left(x {  v}^{-1}\right)\leq \mathcal{E}^{ {  v}}_n\left(x^{p-1} {  v}^{1-p}\right)^{1/(p-1)}$ which follows from the operator concavity of the function $t\mapsto t^{1/(p-1)}$ (here we use the assumption $p\geq 2$), we obtain
\begin{align*}
[ {  w}\,]^{-\frac{1}{p-1}}_{A_p}x_n &\leq  \Biggl(\left(\mathcal{E}_n\left( {  w}\right)\right)^{-1}\left(\mathcal{E}_n\left( {  v}\right)\right)^{1-p}\left(\mathcal{E}_n\left(x\right)\right)^{p-1}\Biggr)^{\frac{1}{p-1}} \\
&=  \left(\mathcal{E}_n\left( {  w}\right)\right)^{-\frac{1}{p-1}}\,\mathcal{E}^{ {  v}}_n\left(x {  v}^{-1}\right) \\
& \leq  \left(\mathcal{E}_n\left( {  w}\right)\right)^{-\frac{1}{p-1}}\left(\mathcal{E}^{ {  v}}_n\left(x^{p-1} {  v}^{1-p}\right)\right)^\frac{1}{p-1}.
\end{align*}
However, by the definition of $a$ and the operator monotonicity of the function $t\mapsto t^{1/(p-1)}$ (again, here we use the assumption $p\geq 2$) we get
$$\Big(\mathcal{E}^{ {  v}}_n\left(x^{p-1} {  v}^{1-p}\right)\Big)^\frac{1}{p-1}=\Big(\mathcal{E}_n\Big[\mathcal{E}^{ {  v}}_n\left(x^{p-1} {  v}^{1-p}\right)\Big]\Big)^\frac{1}{p-1}\leq \big(\mathcal{E}_n(a)\big)^\frac{1}{p-1}.$$
Therefore, we can proceed with the previous bound as follows:
\begin{align*}
[ {  w}\,]^{-\frac{1}{p-1}}_{A_p}x_n &\leq   \Big(\mathcal{E}_n\left( {  w}\right)^{-1} \mathcal{E}_n(a)\Big)^{\frac{1}{p-1}}= \Big(\mathcal{E}_n^{ {  w}}\left(a {  w}^{-1}\right)\Big)^{\frac{1}{p-1}} \leq b^{\frac{1}{p-1}},
\end{align*}
where the last bound is due to the definition of $b$ and the operator monotonicity of $t\mapsto t^{1/(p-1)}$. Thus we have obtained the majorant $[ {  w}\,]^{\frac{1}{p-1}}_{A_p}b^{\frac{1}{p-1}}$ for the nonnegative martingale $(x_n)_{n \geq 0}$, and it remains to apply \eqref{bounda1} and \eqref{boundb1} to get
\begin{displaymath}
\|b^{\frac{1}{p-1}}\|_{L_p^ {  w}\left(\mathcal{M}\right)} = \|b\|_{L_{p'}^ {  w}\left(\mathcal{M}\right)}^{\frac{1}{p-1}} \leq c_{p}^{\frac{1}{p-1}}\|a\|^{\frac{1}{p-1}}_{L_{p'}^ {  v}\left(\mathcal{M}\right)} \leq  c_{p}^\frac{2}{p-1}\|x\|_{L_{p}^ {  w}\left(\mathcal M\right)}.
\end{displaymath}
That is, we have found the majorant of $(x_n)_{n\geq 0}$ whose $L_p^ { {w}}$ norm is bounded by $c_{p}^{\frac{2}{p-1}}[ {  w}\,]^{\frac{1}{p-1}}_{A_p}\|x\|_{L_p^ { {w}}(\mathcal{M})}$, as desired.
\end{proof}

We are ready to complete the proof of Theorem \ref{doob2}.

\begin{proof}[Proof of Theorem \ref{doob2} for $1<p<2$] Again, we proceed by duality. Fix a weight $ {  w}\in A_p$, an arbitrary finite sequence $(a_n)_{n\geq 0}$ of positive operators contained in $L_p^ {  w}(\mathcal{M})$ and any $g\in L_{p'}^ {  v}(\mathcal{M})$ of norm one. By Theorem \ref{ref}, there exists a majorant $b$ of the  martingale $(\mathcal{E}_n(g))_{n\geq 0}$, satisfying $\|b\|_{L_{p'}^ {  v}(\mathcal{M})}\leq c_{p'}[v]_{A_{p'}}^{1/(p'-1)}=c_{p'}[w]_{A_p}$. Therefore, by H\"older's inequality,
$$ \tau\left(\sum_{n\geq 0} \mathcal{E}_n\left(a_n\right)g\right)=\tau\left(\sum_{n\geq 0}a_n\mathcal{E}_n(g)\right)\leq \tau\left(\sum_{n\geq 0}a_n b\right)\leq c_{p'}[w]_{A_p}\left\|\sum_{n \geq 0}a_n\right\|_{L_p^ {  w}(\mathcal{M})}.$$
The proof is completed by taking the supremum over all $g$ as above.
\end{proof}

The above proof works without any assumption on the regularity of the filtration. We would like to conclude this section by an example showing that in this general context, the standard self-improving properties and reverse H\"older inequalities may fail for $A_p$ weights.

\begin{remark}\label{nonhomog}
Consider the sequence  $a_n=2^{-n}(n!)^{-1}$, $n=0,\,1,\,2,\,\ldots$.
On the (commutative) probability space $([0,1],\mathcal{B}([0,1]),|\cdot|)$, consider the filtration $(\F_n)_{n\geq 0}$, where $\F_n$ is generated by the intervals $[0,a_n]$, $(a_n,a_{n-1}]$, $(a_{n-1},a_{n-2}]$, $\ldots$, $(a_1,a_0]$. Let $w$ be the weight given by $w=\sum_{n=0}^\infty n! \chi_{(a_{n+1},a_n]}$. This is an $A_1$ weight with $[w]_{A_1}\leq 2$: indeed, all the atoms of the filtration are of the form $(a_{n+1},a_n]$ or $[0,a_n]$ for some $n\geq 0$, and
\begin{align*}
 \frac{1}{|(a_{n+1},a_n]|}\int_{a_{n+1}}^{a_n} w\mbox{d}x\leq \frac{1}{|[0,a_n]|}\int_0^{a_n} w\mbox{d}x&=2^n\cdot n!\sum_{k=n}^\infty k! \frac{2k+1}{2^{k+1}(k+1)!}\\
&\leq 2^n\cdot n!\sum_{k=n}^\infty \frac{1}{2^k}=2\cdot n!=2\operatorname*{essinf}_{[0,a_n]}w=2\operatorname*{essinf}_{(a_{n+1},a_n]}w.
\end{align*}
Furthermore, it is evident that for any $\alpha>1$, the function $w^\alpha$ is not integrable: the series
$$ \sum_{n=0}^\infty \frac{(n!)^\alpha(2n+1)}{2^{n+1}(n+1)!}$$
diverges. Therefore, $w$ cannot satisfy reverse H\"older inequality. Similarly, the self-improvement property does not hold. Given any $1<p<\infty$, we know that $w\in A_{p'}$ (since $A_{p'}\subset A_1$) and hence the dual weight $v=w^{1/(1-p')}=w^{1-p}$ belongs to $A_p$. However, if $v$ lied in $A_{p-\e}$ for some $\e>0$, then $v^{1/(1-p+\e)}=w^{(1-p)/(1-p+\e)}$ would be integrable, a contradiction.
\end{remark}

\subsection{A weighted weak-type bound}

The following weighted weak-type inequality is inspired by the  result due to Cuculescu \cite{cucu}. As we mentioned before, the projection $I-q_\lambda$ plays the role of the indicator function of the event $\{\sup_{n\geq 0}|x_n|\geq \lambda\}$. Therefore, this result can be regarded as a noncommutative probabilistic version of \eqref{weightweak}.

\begin{theorem}\label{weak_theorem}
Let $1\leq p<\infty$ and $ {w} \in A_p$. Then for any positive $x \in L_p^ {w}\left(\mathcal{M}\right)$ and any $\lambda > 0$ there exists a projection $q \in \mathcal{M}$ such that $q\mathcal{E}_n\left(x\right)q \leq \lambda$ for all $n \geq 0$ and
\begin{equation}\label{weak_type}
\lambda\Big[\tau^ {w}\left(I-q\right)\Big]^{1/p} \leq [ {w}]_{A_p}^{1/p}\|x\|_{L_p^ {w}\left(\mathcal M\right)}.
\end{equation}
The dependence $[ {w}]_{A_p}^{1/p}$ on the characteristic cannot be improved (i.e., the exponent $1/p$ cannot be decreased) already in the commutative case.
\end{theorem}

\begin{proof}
We study the case $p>1$ only; the argument in the boundary case $p=1$ is analogous, and we leave the details to the reader.
By homogeneity, it is enough to consider the case $\lambda = 1$. Futhermore, using the $\sigma$-finiteness of $(X,\F_0,\mu)$, we may assume that $\mu(X)<\infty$. We recall the construction of Cuculescu's projections: let $q_{-1}$ = $I$ and for $n \geq 0$ define $q_n$ inductively by the equation
\begin{displaymath}
q_n = q_{n-1}{I}_{\left[0,1\right]}\left(q_{n-1}\mathcal{E}_n\left(x\right)q_{n-1}\right).
\end{displaymath}

The sequence $\left(q_n\right)_{n\geq -1}$ is nonincreasing and it enjoys following properties (for detailed proofs, see \cite{cucu} or \cite{rand}):
\begin{enumerate}[\rm (i)]
\item for every $n \geq 0$, $q_n \in \mathcal{M}_n$;
\item $q_n$ commutes with $q_{n-1}\mathcal{E}_n\left(x\right)q_{n-1}$;
\item $q_n\mathcal{E}_n\left(x\right)q_n \leq q_n$;
\item $\left(q_{n-1}-q_n\right)\mathcal{E}_n\left(x\right)\left(q_{n-1}-q_n\right) \geq q_{n-1}-q_n$.
\end{enumerate}
Set $q = \bigwedge_{n=0}^\infty q_n$. Then $q\mathcal{E}_n\left(x\right)q = qq_n\mathcal{E}_n\left(x\right)q_nq  \leq qq_nq \leq I$, so by the above properties, we obtain
\begin{align*}
\tau^ {  w}\left(I-q_n\right) = \sum_{k=0}^n \tau\left(\left(q_{k-1}-q_k\right) {  w}\right) &\leq \sum_{k=0}^n \tau\left(\left(q_{k-1}-q_k\right)\mathcal{E}_k\left(x\right)\left(q_{k-1}-q_k\right) {  w}\right) \\
& = \sum_{k=0}^n \tau\left(\left(q_{k-1}-q_k\right)x\left(q_{k-1}-q_k\right)\mathcal{E}_k\left( {  w}\right)\right)\\
& =  \tau\left(\sum_{k=0}^n \left(q_{k-1}-q_k\right)\mathcal{E}_k\left( {  w}\right) {  w}^{-\frac{1}{p}}x {  w}^{\frac{1}{p}}\right),
\end{align*}
where in the last line we have exploited the tracial property and commuting of $ {  w}$ with all elements of $\mathcal{M}$.
By H\"older's inequality, mutual orthogonality of projections $\left(\left(q_{k-1}-q_k\right)\right)_{k \geq 0}$ and the definition of $[w]_{A_p}$, we may proceed as follows (recall that $ {  v}= {  w}^{1-p'}$ is the dual weight of $ {  w}$):
\begin{align*}
\tau^ {  w}\left(I-q_n\right) &\leq \tau\left(\left(\sum_{k=0}^n \left(q_{k-1}-q_k\right)\mathcal{E}_k\left(w\right)\right)^{p'} {  w}^{-\frac{1}{p-1}}\right)^{\frac{1}{p'}}\tau\left(x^p {  w}\right)^{\frac{1}{p}}\\
& = \left(\sum_{k=0}^n\tau\left( \left(q_{k-1}-q_k\right)\left(\mathcal{E}_k\left( {  w}\right)\right)^{p'} {  v}\right)\right)^{\frac{1}{p'}}\|x\|_{L_p^ {  w}\left(\mathcal{M}\right)}\\
& = \left(\sum_{k=0}^n\tau\left( \left(q_{k-1}-q_k\right)\mathcal{E}_k\left( {  w}\right)\left(\mathcal{E}_k\left( {  w}\right)\right)^{\frac{1}{p-1}}\mathcal{E}_k\left( {  v}\right)\right)\right)^{\frac{1}{p'}}\|x\|_{L_p^ {  w}\left(\mathcal{M}\right)}\\
& \leq \left([ {  w}]_{A_p}^{\frac{1}{p-1}}\sum_{k=0}^n\tau\left(\left(q_{k-1}-q_k\right)\mathcal{E}_k\left( {  w}\right)\right)\right)^{\frac{1}{p'}}\|x\|_{L_p^ {  w}\left(\mathcal{M}\right)}\\
& = [ {  w}]_{A_p}^{\frac{1}{p}} \left(\sum_{k=0}^n \tau\left(\left(q_{k-1}-q_k\right) {  w}\right)\right)^{\frac{1}{p'}} \|x\|_{L_p^ {  w}\left(\mathcal{M}\right)}\\
& = [ {  w}]_{A_p}^{\frac{1}{p}} \left( \tau^ {  w}\left(I-q_n\right)\right)^{\frac{1}{p'}} \|x\|_{L_p^ {  w}\left(\mathcal{M}\right)}.
\end{align*}
But the assumption $\mu(X)<\infty$ implies $w\in L_1$ and hence the trace $\tau^ {  w}\left(I-q_n\right)$ is finite. Therefore, multiplying both sides by $\left( \tau^ {  w}\left(I-q_n\right)\right)^{-\frac{1}{p'}}$ we obtain $\left( \tau^ {  w}\left(I-q_n\right)\right)^{\frac{1}{p}} \leq [ {  w}]_{A_p}^{\frac{1}{p}} \|x\|_{L_p^ {  w}\left(\mathcal{M}\right)}$, and passing with $n \to \infty$ gives the weighted weak type inequality.
\end{proof}

\begin{remark}\label{sigma-fin}
In our considerations below, we will need versions of Theorem \ref{ref} and Theorem \ref{weak_theorem} for filtrations indexed by $\mathbb{Z}$. Using \eqref{not_difficult}, one easily obtains these statements under the assumption that $(X,\F_n,\mu)$ is $\sigma$-finite for each $n$.
\end{remark}

\section{Maximal inequalities on metric spaces}
In this section, as an application of Theorem \ref{ref}, we establish the noncommutative weighted Hardy-Littlewood maximal inequalities on metric spaces. These results can be considered as noncommutative version of \eqref{weightlp} and \eqref{weightweak}. In particular,  Mei's results \cite[Chapter 3]{Mei}  are extended to the weighted case.

Suppose that $(X,d)$ is a metric space equipped with the $\sigma$-field of its Borel subsets $\F$ and a Radon measure $\mu$. The symbol $B(x,r)=\{y\in X\,:\,d(y,x)\leq r\}$ stands for the closed ball of center $x$ and radius $r$. We assume the non-degeneracy condition $0<\mu(B)<\infty$ for any ball $B$ of positive radius. Furthermore, we will work with measures $\mu$ satisfying the so-called  doubling condition: there exists a finite constant $\kappa$ such that $\mu(B(x,2r))\leq \kappa\mu(B(x,r))$ for all $x\in X$ and $r>0$.

Given $1<p<\infty$ and a weight $w$ on $X$, we say that $w$ satisfies Muckenhoupt's condition $A_p$, if its $A_p$ characteristic
$$ [w]_{A_p}:=\sup_{x\in X,\,r>0} \left(\frac{1}{\mu(B(x,r))}\int_{B(x,r)}w\mbox{d}\mu\right)\left(\frac{1}{\mu(B(x,r))}\int_{B(x,r)}w^{1/(1-p)}\mbox{d}\mu\right)^{p-1}$$
is finite. A weight $w$ belongs to the class $A_1$, if there is a constant $c$ such that for all $r>0$ and all $x\in X$,
$$ \frac{1}{\mu(B(x,r))}\int_{B(x,r)}w\mbox{d}\mu \leq c\operatorname*{essinf}_{B(x,r)}w.$$
The smallest $c$ with the above property is denoted by $[w]_{A_1}$ and called the $A_1$ characteristic of $w$.
 Finally, consider the von Neumann algebra $\mathcal{N}$ and put $\mathcal{M}=L_\infty(X,\F,\mu)\bar{\otimes}\mathcal{N}$. Given $1\leq p<\infty$ and $r>0$, define the averaging operator $\mathcal{A}_r$ acting on locally integrable $f:X\to L_p(\mathcal{N})$ by the formula
$$ \mathcal{A}_rf(x)=\frac{1}{\mu(B(x,r))}\int_{B(x,r)} f\mbox{d}\mu,\qquad x\in X.$$
In particular, if $1\leq p<\infty$ and $w\in A_p$, then $\mathcal{A}_rf$ is well defined for $f\in L_p^w(\mathcal{M})$: any $f\in L_p^w(\mathcal{M})$ is locally integrable as a function from $X$ to $L_1(\mathcal{N})$. Indeed, if $p>1$, then H\"older's inequality gives
$$ \int_{B(x,r)}\|f\|_{L_1(\mathcal{N})}\mbox{d}\mu\leq \|f\|_{L_p^w(\mathcal{M})}\left(\int_{B(x,r)} w^{1/(1-p)}\mbox{d}\mu\right)^{(p-1)/p}<\infty.$$
For $p=1$ the argument is even simpler: $\int_{B(x,r)}\|f\|_{L_1(\mathcal{N})}\mbox{d}\mu\leq \|f\|_{L_1^w(\mathcal{M})}\int_{B(x,r)}w^{-1}\mbox{d}\mu<\infty.$

The main result of this section is stated below. It can be regarded as the noncommutative version of \eqref{weightlp} and \eqref{weightweak}, with the extraction of the optimal dependence on $[w]_{A_p}$.

\begin{theorem}\label{thm_HL}
Let $1\leq p<\infty$ and assume that $w$ is an $A_p$ weight on $X$. Then for any $f\in L_p^w(\mathcal{M})$ and any $\lambda>0$ there is a projection $q\in \mathcal{M}$ satisfying
$$ \lambda \Big[ \tau^ {w}(I-q)\Big]^{1/p}\lesssim_p [w]_{A_p}^{1/p}\|f\|_{L_p^w(\mathcal{M})}$$
and $q\mathcal{A}_rfq\leq \lambda q$ for all $r>0$. Furthermore, if $p>1$, then there exists a constant $c_p$ depending only on $p$ such that for any $f\in L_p(X;L_p(\mathcal{M}))$,
$$ \|(\mathcal{A}_rf)_{r>0}\|_{L_p^w(\mathcal{M},\ell_\infty)}\leq c_p[w]_{A_p}^{1/(p-1)}\|f\|_{L_p^w(\mathcal{M})}.$$
\end{theorem}

Our argument will exploit the following fact proved in \cite[Theorem 4.1]{HK}.

\begin{lemma}\label{partit}
Let $(X,d)$ be the metric space equipped with a Radon measure $\mu$ satisfying the above requirements. Then there exist a constant $C$ and a finite collection of families $\mathcal{P}^1$, $\mathcal{P}^2$, $\ldots$, $\mathcal{P}^N$, where each $\mathcal{P}^k=(\mathcal{P}^k_j)_{j\in \mathbb{Z}}$ is a sequence of partitions of $X$, such that the following holds.

\begin{enumerate}[\rm (i)]
\item For each $1\leq k\leq N$ and each $j\in \mathbb{Z}$, the partition $\mathcal{P}^{k}_{j+1}$ is a refinement of $\mathcal{P}^k_j$.

\item For all $x\in X$ and $r>0$, there is $1\leq k\leq N$, $j\in \mathbb{Z}$ and an element $Q\in \mathcal{P}^k_j$ such that $B(x,r)\subseteq Q$ and $\mu(Q)\leq C\mu(B(x, r))$.

\item Any $Q\in \bigcup_{k,j}\mathcal{P}^k_j$ is contained within some ball $B(x,r)$ such that $\mu(B(x,r))\leq C\mu(Q)$.
\end{enumerate}
\end{lemma}

\begin{proof}[Proof of Theorem \ref{thm_HL}]By a standard decomposition argument, we may assume that $f$ is nonnegative: we have $f(\omega)\geq 0$ for any $\omega\in X$. Let $N$ be the number guaranteed by the above lemma and fix $k\in \{1,2,\ldots,N\}$. For $n\in \mathbb{Z}$, let $\mathfrak{F}^k_n$ be the $\sigma$-field generated by $\mathcal{P}^k_n$ and denote by $\mathcal{E}^k_n$ the associated conditional expectation. Note that the martingale $w=(\mathcal{E}_n^kw)_{n\in \mathbb{Z}}$ satisfies the $A_p$ condition: by Lemma \ref{partit} (iii), for any $Q\in \bigcup_{n\in \mathbb{Z}} \mathcal{P}^k_n$ we have
\begin{align*}
& \left(\frac{1}{\mu(Q)}\int_Q w\right)\left(\frac{1}{\mu(Q)}\int_Q w^{1/(1-p)}\right)^{p-1}\\
&\leq C^p \left(\frac{1}{\mu(B(x,r))}\int_{B(x,r)}w\mbox{d}\mu\right)\left(\frac{1}{\mu(B(x,r))}\int_{B(x,r)}w^{1/(1-p)}\mbox{d}\mu\right)^{p-1}\leq C^p[w]_{A_p},
\end{align*}
where $B(x,r)$ is the ball containing $Q$. An analogous argument works for $p=1$.
 By Theorem \ref{weak_theorem} and Remark \ref{sigma-fin}, applied to the martingale $f=(\mathcal{E}^k_nf)_{n\in \mathbb{Z}}$, for any $\lambda>0$ there exists a projection $q_k$ such that $q_k\mathcal{E}_n^kfq_k\leq \lambda$ for all $n\in \mathbb{Z}$ and $\lambda \tau^{w}(I-q_k)^{1/p}\leq C[w]_{A_p}^{1/p}\|f\|_{L_p^w(\mathcal{M})}$. Take $q=\bigwedge_{k=1}^N q_k $, the projection onto the intersection $\bigcap q_k(H)$. Since $I-\bigwedge_{k=1}^N q_k\leq \sum_{k=1}^N (I-q_k)$, we get
 $$ \lambda \big[\tau^w(I-q)\big]^{1/p}\leq N^{1/p}C[w]_{A_p}^{1/p}\|f\|_{L_p^w(\mathcal{M})}.$$
 Now we apply the second part of Lemma \ref{partit}: given an arbitrary ball $B(x,r)$, there is an associated set $Q$, belonging to some $\mathcal{P}^k_n$. Therefore,
 \begin{equation}\label{bodd}
 \mathcal{A}_rf(x)=\frac{1}{B(x,r)}\int_{B(x,r)}f\mbox{d}\mu\leq \frac{C}{\mu(Q)}\int_Q f\mbox{d}\mu=C\mathcal{E}^k_nf(x)
\end{equation}
 and consequently $q\mathcal{A}_rfq\leq C\lambda$ for all $r$. This proves the weighted weak-type inequality for $(\mathcal{A}_rf)_{r>0}$. Concerning the strong-type estimate, note that \eqref{bodd} yields
\begin{align*}
\|(\mathcal{A}_rf)_{r>0}\|_{L_p^w(\mathcal{M};\ell_\infty)}&\leq C\left\|\left(\sum_{k=1}^N \mathcal{E}^k_nf\right)_{n\in \mathbb{Z}}\right\|_{L_p^w(\mathcal{M};\ell_\infty)}\\
&\leq C\sum_{k=1}^N\left\|\left( \mathcal{E}^k_nf\right)_{n\in \mathbb{Z}}\right\|_{L_p^w(\mathcal{M};\ell_\infty)}\leq C'N[w]_{A_p}^{1/(p-1)}\|f\|_{L_p^w(\mathcal{M})}.
\end{align*}
This gives the claim.
\end{proof}

\begin{remark}
In particular, one may apply the above estimates in the context when $X$ is a locally compact group $G$, equipped an invariant metric $d$ and the right-invariant Haar measure $m$. The averaging operators
$$ \mathcal{A}_rf(g)=\frac{1}{\mu(B(g,r))}\int_{B(g,r)}f(h)\mbox{d}m(h)=\frac{1}{\mu(B(e,r))}\int_{B(e,r)} f(gh)\mbox{d}m(h)$$
appear naturally in the study of ergodic theorems, concerning the action of amenable groups on noncommutative $L_p$ spaces (cf. \cite{HLW}).
\end{remark}

\section{A weighted inequality for maximal singular integrals}

The next application of Theorem \ref{ref} concerns weighted bounds for maximal singular integrals of operator-valued functions in dimension one. Let us start with some motivation. The Hilbert transform $\mathcal{H}$, the fundamental object in harmonic analysis, is an operator which acts on locally integrable functions $f:\R\to \R$ by
$$ \mathcal{H}f(s)=\mbox{p.v.}\frac{1}{\pi}\int_\R \frac{f(t)}{s-t}\mbox{d}t.$$
Here `p.v.' refers to the principal value of the integral: $ \mathcal{H}f(s)=\lim_{\e\downarrow 0} \mathcal{H}^{\e}f(s),$ and
$$ \mathcal{H}^\e f(s)=\frac{1}{\pi}\int_{|s-t|>\e} \frac{f(t)}{s-t}\mbox{d}t$$
is the truncated Hilbert transform. The above limiting procedure makes sense for certain vector-valued functions as well: one can define $\mathcal{H}f$ for $f$ taking values in the so-called UMD Banach spaces. Recall that a Banach space $\mathbb{B}$ is UMD (Unconditional for Martingale Differences), if the following holds. For any (equivalently, for all) $1<p<\infty$, there exists a finite constant $c_{p,\mathbb{B}}$ such that for any (classical, commutative) martingale difference $d=(d_k)_{k\geq 0}$ with values in $\mathbb{B}$, given on some filtered probability space $(\Omega,\F,(\F_k)_{k\geq 0},\mathbb{P})$, and any deterministic sequence $\e=(\e_k)_{k\geq 0}$ with values in $[-1,1]$ we have
$$ \left\|\sum_{k=0}^n \e_kd_k\right\|_{L_p(\Omega;\mathbb{B})}\leq c_{p,\mathbb{B}}\left\|\sum_{k=0}^n d_k\right\|_{L_p(\Omega;\mathbb{B})},\qquad n=0,\,1,\,2,\,\ldots.$$
Here the probability space, as well as the filtration, are allowed to vary. Note that for any $1<p<\infty$ and any von Neumann algebra $\mathcal{N}$, the space $L_p(\mathcal{N})$ is UMD: this follows directly from \eqref{Lp}, applied to $\mathcal{M}=L_\infty(\Omega,\F,\mathbb{P})\overline{\otimes} \mathcal{N}$.
Next, a well-known result of Burkholder \cite{B1.1} asserts that if $\mathbb{B}$ is a UMD space, then $\|\mathcal{H}\|_{L_p(\R;\mathbb{B})\to L_p(\R;\mathbb{B})}\lesssim c_{p,\mathbb{B}}^2.$ Putting all the above facts together, we see that the action of the Hilbert transform on $L_p(\mathcal{M})=L_p(L_\infty(\R)\overline{\otimes} \mathcal{N})$, the space of $L_p(\mathcal{N})$-valued functions on $\R$, is well defined and bounded for $1<p<\infty$. % Actually, it can be shown that  $\|\mathcal{H}\|_{L_p(\mathcal{M})\to L_p(\mathcal{M})}\lesssim (pp')^2.$

One can also study analogous \emph{weighted} $L_p$ estimates for martingale transforms and the Hilbert transform. It follows from the results of Lacey \cite{Lac} that  if $d=(d_k)_{k\geq 0}$ is a martingale difference with values in a UMD space $\mathbb{B}$, $\e=(\e_n)_{n\geq 0}$ is a predictable sequence of signs and $w$ is an $A_p$ weight on $\Omega$, then we have
\begin{equation}\label{w_transform}
 \left\|\sum_{k=0}^n v_kd_k\right\|_{L_p^w(\Omega;\mathbb{B})}\lesssim [w]_{A_p}^{\max\{1/(p-1),1\}}\left\|\sum_{k=0}^n d_k\right\|_{L_p^w(\Omega;\mathbb{B})}.
\end{equation}
Here $\|f\|_{L_p^w(\Omega;\mathbb{B})}=\left(\int_\Omega \|f\|_\mathbb{B}^p\mbox{d}\mathbb{P}\right)^{1/p}$. Moreover, we have $\|\mathcal{H}\|_{L_p^w(\R;\mathbb{B})\to L_p^w(\R;\mathbb{B})}\lesssim [w]_{A_p}^{\max\{(p-1)^{-1},1\}}$. The exponent ${\max\{(p-1)^{-1},1\}}$ is optimal in both estimates above. In particular, specifying $\mathbb{B}=L_p(\mathcal{N})$ and $\mathcal{M}=L_\infty(\R)\overline{\otimes} \mathcal{N}$, as above, we get the corresponding version for noncommutative martingale transforms and
\begin{equation}\label{Ha}
\|\mathcal{H}\|_{L_p^w(\mathcal{M})\to L_p^w(\mathcal{M})}\lesssim [w]_{A_p}^{\max\{(p-1)^{-1},1\}}.
\end{equation}
There is another, related operator, playing an important role in harmonic analysis: the so-called maximal truncation $\mathcal{H}^*$, given by $\mathcal{H}^*f=\sup_{\e>0}|\mathcal{H}^\e f|$. This operator also satisfies the weighted bound \eqref{Ha}, which can be handled with the use of Cotlar's inequality or the direct majorization in terms of sparse operators (see \cite{Lac}). Both these approaches exploit a number of pointwise estimates which cannot be used in the noncommutative context.

The purpose of this section is to establish a noncommutative maximal version of \eqref{Ha} for maximal truncation (with a slightly worse dependence on $[w]_{A_p}$). On the positive side,  we will work in the more general class of convolution-type singular integrals on $\R$. Throughout, we assume that $K:(-\infty,0)\cup(0,\infty)\to \R$ is an odd, twice differentiable function (in the sense that $K'$ is absolutely continuous) which satisfies
\begin{equation}\label{cond1}
 \lim_{s\to \infty} K(s)=\lim_{s\to \infty}K'(s)=0
\end{equation}
and
\begin{equation}\label{cond2}
 s^3K''(s)\in L^\infty(\R).
\end{equation}
We denote by $T_K$ the associated one-dimensional singular integral operator, defined by
$$ T_Kf(s)=\operatorname{p.v.}\int_\R f(t)K(s-t)\mbox{d}t=\lim_{\e\downarrow 0}T_K^\e f(s),$$
where $T_K^\e f(s)$ is the truncation at level $\e$:
$$ T_K^\e f(s)=\int_{|s-t|>\e} f(t)K(s-t)\mbox{d}t.$$
In analogy to the above setting, we may also introduce the maximal truncation $T^*_K$ by $T^*_Kf=\sup_{\e>0}|T^\e_Kf|$. In all the above definitions, $f$ is allowed to be vector-valued.
Note that the choice $K(s)=1/(\pi s)$ brings us back to the context of Hilbert transform.

As shown by Vagharshakyan \cite[Theorem 2.4]{Va}, the operator $T_K$ can be expressed as an average of appropriate one-dimensional dyadic shifts. To recall the necessary definitions, let $\varphi$, $\psi:\R\to \R$ be two functions  supported on the unit interval $[0,1]$ and given there by the formulas
$$ \varphi(x)=\begin{cases}
-1 & \mbox{if }0\leq x<1/4,\\
1 & \mbox{if }1/4\leq x< 3/4,\\
-1 & \mbox{if }3/4\leq x\leq 1
\end{cases}
\qquad \mbox{and}\qquad  \psi(x)=\begin{cases}
7 &\mbox{if }0<x<1/4,\\
-1 & \mbox{if }1/4\leq x< 1/2,\\
1 & \mbox{if }1/2\leq x<3/4,\\
-7 & \mbox{if }3/4\leq x\leq 1.
\end{cases}$$
For any (real or vector-valued) function $f$ on $\R$ and any interval $I=[a,b]$, we define the scaled function $f_I$ by
$$ f_I(x)=\frac{1}{\sqrt{b-a}}f\left(\frac{x-a}{b-a}\right),\qquad x\in \R.$$
For any $\beta=\{\beta_l\}\in \{0,1\}^\mathbb{Z}$ and any $r\in [1,2)$, we define the dyadic grid $\mathbb{D}_{r,\beta}$ to be the following collection of intervals (see \cite{NTV} for the motivation and basic properties of this family):
$$ \mathbb{D}_{r,\beta}=\left\{ r2^n\left([0,1)+k+\sum_{i<n} 2^{i-n}\beta_i\right)\right\}_{n\in\mathbb{Z},k\in \mathbb{Z}}.$$
We equip $\{0,1\}^\mathbb{Z}$ with the uniform probability measure $\mu$, uniquely determined by the requirement
$$ \mu(\{\beta:(\beta_{i_1},\beta_{i_2},\ldots,\beta_{i_n})=a\})=2^{-n}$$
for any $n$, any sequence $i_1<i_2<\ldots<i_n$ of integers and any $a\in \{0,1\}^n$.

The aforementioned result of Vagharshakyan asserts the following.

\begin{theorem}\cite[Theorem 2.4]{Va}
Suppose that the kernel $K$ satisfies \eqref{cond1} and \eqref{cond2}. Then there exists a coefficient function $\gamma:(0,\infty)\to \R$ satisfying
$$ \|\gamma\|_\infty\leq C\|s^2K''(s)\|_\infty$$
such that
\begin{equation}\label{defK}
 K(t-s)=\int_{\{0,1\}^\mathbb{Z}}\int_1^2 \sum_{I\in \mathbb{D}_{r,\beta}} \gamma(|I|)\varphi_I(s)\psi_I(t)\frac{\mbox{d}r}{r}\mbox{d}\mu(\beta)
\end{equation}
for all $s\neq t$. Here $C$ is some absolute constant and the series on the right is absolutely convergent almost everywhere.
\end{theorem}

In other words, $T_K$ can be expressed as an average of the Haar shift operators
$$ T_{r,\beta}f=\sum_{I\in \mathbb{D}_{r,\beta}} \gamma(|I|)\langle f,\varphi_I\rangle \psi_I,$$
where $\langle f,g\rangle=\int_\R fg$. Such objects can be handled with the use of martingale methods.

We are ready to establish the main result of this section. In what follows, $\mathcal{M}$ is the von Neumann algebra $L_\infty(\R)\overline\otimes \mathcal{N}$, and hence $L_p^w(\mathcal{M})$  can be identified with the class of appropriately integrable operator-valued functions on $\R$.

\begin{theorem}
For any $1<p<\infty$ and any kernel $K$ satisfying the above assumptions and any $A_p$ weight $w$ on the real line, we have the estimate
$$ \|(T_K^\e)_{\e>0}\|_{L_p^w(\mathcal{M};\ell_\infty)}\leq \tilde{C}_p[w]_{A_p}^{1/(p-1)+\max\{1/(p-1),1\}}\|f\|_{L_p^w(\mathcal{M})}.$$
\end{theorem}
\begin{proof}
%The optimality of the exponent is clear: it is already the best in the non-maximal case (for the Hilbert transform, corresponding to the choice $K(x)=1/x$). By standard approximation, we may and do assume that $f$ is supported on some interval $J$. Furthermore, we may assume that  $\int_\R f=0$, adding to $f$, if necessary, the ``correction function'' of the form $a\chi_{[M,M+1]}$ and then sending $M$ to infinity.
Fix $\e>0$ and $f\in L_p^w(\mathcal{M})$: we may treat it as a function on $\R$ with values in $L_0(\mathcal{N})$. Take two real numbers $s$, $t$ satisfying $|s-t|>\e$. Since both $\varphi_I$ and $\psi_I$ are supported on $I$, we see that $\varphi_I(s)\psi_I(t)=0$ if $|I|\leq \e$ and we may rewrite the identity \eqref{defK} in the form
$$ K(t-s)=\int_{\{0,1\}^\mathbb{Z}}\int_1^2 \sum_{I\in \mathbb{D}_{r,\beta}, |I|> \e} \gamma(|I|)\varphi_I(s)\psi_I(t)\frac{\mbox{d}r}{r}\mbox{d}\mu(\beta).$$
Therefore we have
$$ T_K^\e f(s)=\int_{\{0,1\}^\mathbb{Z}}\int_1^2 \sum_{I\in \mathbb{D}_{r,\beta}, |I|>\e} \gamma(|I|)\langle f,\varphi_I\rangle \psi_I(s)\frac{\mbox{d}r}{r}\mbox{d}\mu(\beta)$$
and hence by Minkowski's inequality,
\begin{align*}
 \|(T_K^\e f)_{\e>0}\|_{L_p^w(\mathcal{M};\ell_\infty)}&\leq \int_{\{0,1\}^\mathbb{Z}}\int_1^2 \left\|\left(\sum_{I\in \mathbb{D}_{r,\beta}, |I|>\e} \gamma(|I|)\langle f,\varphi_I\rangle \psi_I\right)_{\e>0}\right\|_{L_p^w(\mathcal{M};\ell_\infty)}\frac{\mbox{d}r}{r}\mbox{d}\mu(\beta)\\
 &\leq \int_{\{0,1\}^\mathbb{Z}}\int_1^2 \left\|\left(\sum_{I\in \mathbb{D}_{r,\beta}, |I|>\e,\atop  I\text{ odd}} \gamma(|I|)\langle f,\varphi_I\rangle \psi_I\right)_{\e>0}\right\|_{L_p^w(\mathcal{M};\ell_\infty)}\frac{\mbox{d}r}{r}\mbox{d}\mu(\beta)\\
 &\quad +\int_{\{0,1\}^\mathbb{Z}}\int_1^2 \left\|\left(\sum_{I\in \mathbb{D}_{r,\beta},  |I|>\e, \atop I\text{ even}} \gamma(|I|)\langle f,\varphi_I\rangle \psi_I\right)_{\e>0}\right\|_{L_p^w(\mathcal{M};\ell_\infty)}\frac{\mbox{d}r}{r}\mbox{d}\mu(\beta).
 \end{align*}
 Here and below, $I\in \mathbb{D}_{r,\beta}$ is called odd (even), if so is the number $\log_2(|I|/r)$.
From now on, we will restrict our analysis to `even sums' only; the first summand in the last line above can be dealt with analogously. The sequence
$$ \sum_{I\in \mathbb{D}_{r,\beta}, |I|\geq 4^{n},\atop  I\text{ even}} \gamma(|I|)\langle f,\varphi_I\rangle \psi_I,\qquad n\in \mathbb{Z},$$
is a martingale with respect to its natural filtration. It is crucial here that we assume the `double spread' on $\log_2(|I|/r)$ (i.e., we assume that $\log_2(|I|/r)$ has the fixed parity): thanks to this condition, $(\sum_{I\in \mathbb{D}_{r,\beta}, |I|= r4^{n}} \gamma(|I|)\langle f,\varphi_I\rangle \psi_I)_{n\in \mathbb{Z}}$ is a martingale difference sequence. The application of Theorem \ref{ref} yields
$$ \left\|\left(\sum_{I\in \mathbb{D}_{r,\beta}, |I|>\e,\atop  I\text{ even}} \gamma(|I|)\langle f,\varphi_I\rangle \psi_I\right)_{\e>0}\right\|_{L_p^w(\mathcal{M};\ell_\infty)}\leq C_p [w]_{A_p}^{1/(p-1)}\left\|\sum_{I\in \mathbb{D}_{r,\beta},\atop  I\text{ even}} \gamma(|I|)\langle f,\varphi_I\rangle \psi_I\right\|_{L_p^w(\mathcal{M})}.$$
The next step is to prove that the right-hand side is controlled by $\|f\|_{L_p^w(\mathcal{M})}$. This will follow from the Theorem \ref{UMD_theorem} below.
\end{proof}

From now on, we move to the classical context; all the functions and processes considered below are commutative.

\begin{theorem}\label{UMD_theorem}
Suppose that $\mathbb{B}$ is a UMD space and $f:\R\to \mathbb{B}$ is a Bochner integrable function. Then for $1<p<\infty$ and any $A_p$ weight $w$ on $\R$ we have
\begin{equation}\label{weightedd}
 \left\|\sum_{I\in \mathbb{D}_{r,\beta},\atop  I\text{ even}} \gamma(|I|)\langle f,\varphi_I\rangle \psi_I\right\|_{L_p^w(\R;\mathbb{B})}\leq C_p\|\gamma\|_\infty [w]_{A_p}^{\max\{1/(p-1),1\}}\|f\|_{L_p^w(\R;\mathbb{B})}.
 \end{equation}
The same estimate holds if the sum on the left is taken over $I\in \mathbb{D}_{r,\beta}$ with odd $I$.
\end{theorem}

This result follows from Theorem 5.1 in \cite{HPTV} in the real-valued case, the context of $UMD$ spaces requires more effort. We guess that even in the vector setting the result is known, however, the proof presented below will use a number of novel arguments from martingale theory. We will exploit the so-called sparse operators, which have gained a lot of interest in the recent literature (we mention here the convenient references: Domelevo and Petermichl \cite{DP}, Lerner \cite{Le2}, Lorist \cite{Lo}, which contain argumentation related to that below). Let us briefly outline our approach.  The idea is to control pointwise the sum in \eqref{weightedd} by a similar expression, in which the summation is taken over much smaller collection of intervals, satisfying the so-called \emph{sparseness} condition (the formal definitions will appear later). The proof of such a domination rests on an unweighted, weak-type version of \eqref{weightedd}, which will be obtained with the use of classical martingales; having established the control, one shows the weighted estimate by a change-of-measure argument, similar to that used in Section 3.

We proceed to the formal analysis. The starting point is the following $L^p$ bound, the unweighted version of Theorem \ref{UMD_theorem}.

\begin{theorem}\label{UMD_theorem2}
Suppose that $\mathbb{B}$ is a UMD space and $f:\R\to \mathbb{B}$ is a Bochner integrable function. Then for $1<p<\infty$ and any bounded sequence $(\gamma(I))_{I\in \mathbb{D}_{r,\beta}}$ we have
$$ \left\|\sum_{I\in \mathbb{D}_{r,\beta},\atop  I\text{ even}} \gamma(I)\langle f,\varphi_I\rangle \psi_I\right\|_{L_p(\R;\mathbb{B})}\leq C_p\|\gamma\|_\infty \|f\|_{L_p(\R;\mathbb{B})}.$$
The same estimate holds if the sum on the left is taken over $I\in \mathbb{D}_{r,\beta}$ with odd $I$.
\end{theorem}
\begin{proof}
We will apply three times the $L^p$ estimate  for martingale transforms, with respect to different filtrations.

\smallskip

 \emph{Step 1.} Consider the truncated version of $\psi$, given by
$$ \zeta(x)=\begin{cases}
1 &\mbox{if }0<x<1/4,\\
-1 & \mbox{if }1/4\leq x< 1/2,\\
1 & \mbox{if }1/2\leq x<3/4,\\
-1 & \mbox{if }3/4\leq x\leq 1.
\end{cases}$$
Then for any even integers $b<c$ we have
\begin{equation}\label{first_step}
\begin{split}
 &\left\|\sum_{I\in \mathbb{D}_{r,\beta},2^b\leq |I|/r\leq 2^c,\atop  I\text{ even}} \gamma(I)\langle f,\varphi_I\rangle \psi_I\right\|_{L_p(\mathcal{M})}
\leq c_p\left\|\sum_{I\in \mathbb{D}_{r,\beta},2^b\leq |I|/r\leq 2^c,\atop  I\text{ even}} \gamma(I)\langle f,\varphi_I\rangle \zeta_I\right\|_{L_p(\mathcal{M})}.
\end{split}
\end{equation}
To see this, split the function $\psi$ into two, the outer and the inner part:
$$ \psi^{\text{out}}(x)=\begin{cases}
7 & \mbox{if }0\leq x<1/4,\\
0 & \mbox{if }1/4\leq x< 3/4,\\
-7 & \mbox{if }3/4\leq x\leq 1
\end{cases}
\qquad \mbox{and}\qquad  \psi^{\text{inn}}(x)=\begin{cases}
0 &\mbox{if }0<x<1/4,\\
-1 & \mbox{if }1/4\leq x< 1/2,\\
1 & \mbox{if }1/2\leq x<3/4,\\
0 & \mbox{if }3/4\leq x\leq 1.
\end{cases}$$
Introduce the corresponding versions for $\zeta$: then $\zeta^{\text{out}}=\psi^{\text{out}}/7$ and $\zeta^{\text{inn}}=\psi^{\text{inn}}$. We have
\begin{equation}\label{mart1}
\sum_{I\in \mathbb{D}_{r,\beta},2^b\leq |I|/r\leq 2^c,\atop  I\text{ even}} \gamma(I)\langle f,\varphi_I\rangle \psi_I=
\sum_{I\in \mathbb{D}_{r,\beta},2^b\leq |I|/r\leq 2^c,\atop  I\text{ even}} \bigg[\gamma(I)\langle f,\varphi_I\rangle \psi_I^{\text{out}}+ \gamma(I)\langle f,\varphi_I\rangle \psi_I^{\text{inn}}\bigg]
\end{equation}
and
\begin{equation}\label{mart2}
\sum_{I\in \mathbb{D}_{r,\beta},2^b\leq |I|/r\leq 2^c,\atop  I\text{ even}} \gamma(I)\langle f,\varphi_I\rangle \zeta_I=
\sum_{I\in \mathbb{D}_{r,\beta},2^b\leq |I|/r\leq 2^c,\atop  I\text{ even}} \bigg[\gamma(I)\langle f,\varphi_I\rangle \zeta_I^{\text{out}}+ \gamma(I)\langle f,\varphi_I\rangle \zeta_I^{\text{inn}}\bigg].
\end{equation}
Since $\psi^{\text{inn}}$, $\psi^{\text{out}}$ have integral zero, the partial sums corresponding to the right-hand sides of \eqref{mart1} and \eqref{mart2} are martingales. Specifically, if $n$ is an even integer between $b$ and $c$, then the $n$-th differences are
$$\sum_{I\in \mathbb{D}_{r,\beta},|I|=r2^n}\gamma(I)\langle f,\varphi_I\rangle \psi_I^{\text{out}}\qquad\mbox{and}\qquad  \sum_{I\in \mathbb{D}_{r,\beta},|I|=r2^n}\gamma(I)\langle f,\varphi_I\rangle \zeta_I^{\text{out}},$$
while for odd $n$ (satisfying $b\leq n-1\leq c$), the differences are
$$\sum_{I\in \mathbb{D}_{r,\beta},|I|=r2^{n-1}}\gamma(|I|)\langle f,\varphi_I\rangle \psi_I^{\text{inn}}\qquad\mbox{and}\qquad  \sum_{I\in \mathbb{D}_{r,\beta},|I|=r2^{n-1}}\gamma(I)\langle f,\varphi_I\rangle \zeta_I^{\text{inn}}.$$
Furthermore, by the above discussion, the martingale associated with \eqref{mart2} is the transform of the martingale in \eqref{mart1} by a predictable sequence with values in $\{1,7\}$. This yields \eqref{first_step}.

\smallskip

\emph{Step 2.} Now we will prove that for any even integers $b<c$ we have
\begin{equation}\label{second_step}
\begin{split}
 &\left\|\sum_{I\in \mathbb{D}_{r,\beta},2^b\leq |I|/r\leq 2^c,\atop  I\text{ even}} \gamma(I)\langle f,\varphi_I\rangle \zeta_I\right\|_{L_p(\mathcal{M})}
\leq c_p\left\|\sum_{I\in \mathbb{D}_{r,\beta},2^b\leq |I|/r\leq 2^c,\atop  I\text{ even}} \gamma(I)\langle f,\varphi_I\rangle \varphi_I\right\|_{L_p(\mathcal{M})}.
\end{split}
\end{equation}
The argument is the same as previously, but we need a different filtration. Namely, we take $ \zeta^{\text{out}}=\zeta \chi_{[0,1/2)}$, $\zeta^{\text{inn}}=\zeta \chi_{[1/2,1)}$ and similarly for $\varphi^{\text{out}}$ and $\varphi^{\text{inn}}$. Then $\zeta^{\text{out}}=-\varphi^{\text{out}}$ and $\zeta^{\text{inn}}=\varphi^{\text{inn}}$, so the corresponding `finer' martingales associated with the left- and the right-hand side of \eqref{second_step} are transforms of each other by a predictable sequence of signs.

\smallskip

 \emph{Step 3.} The final part is to note that
\begin{equation}\label{third_step}
\begin{split}
 &\left\|\sum_{I\in \mathbb{D}_{r,\beta},2^b\leq |I|/r\leq 2^c,\atop  I\text{ even}} \gamma(I)\langle f,\varphi_I\rangle \varphi_I\right\|_{L_p(\mathcal{M})}\leq c_p\|\gamma\|_\infty \|f\|_{L_p(\R;\mathbb{B})}.
\end{split}
\end{equation}
Let $\zeta^{\text{inn}}$ and $\zeta^{\text{out}}$ be the functions introduced in Step 1 above. It is easy to see that the collection $\{\varphi_I,\zeta_I^{\text{inn}},\zeta_I^{\text{out}}\}_{I\in \mathbb{D}_{r,\beta},\,I\text{ even}}$ is a basis in $L_p(\R;\mathbb{B})$ for any fixed $r$ and $\beta$: this is just the Haar basis, under scaling and translation. Expanding $f\in L_p(\R;\mathbb{B})$ into this basis, we get
 $$ f=\sum_{I\in \mathbb{D}_{r,\beta},\,I \text{ even}} \Big(\langle f,\varphi_I\rangle \varphi_I+\langle f,\zeta_I^{\text{inn}}\rangle \zeta_I^{\text{inn}}+\langle f,\zeta_I^{\text{out}}\rangle \zeta^{\text{out}}_I\Big)$$
 and we see that the sum on the left of \eqref{third_step} is obtained by skipping some of the above terms and multiplying the other by the corresponding terms $\gamma(I)$. Thus \eqref{third_step} follows from the $L^p$ estimate for martingale transforms, where the transforming sequence takes values in the set $\{0,\gamma(I)\}_{I\in \mathbb{D}_{r,\beta}}$.

 Putting the above three steps together and letting $b\to-\infty$, $c\to \infty$, we get the desired assertion.
\end{proof}

\begin{remark}
One might repeat the above argumentation, replacing the $L_p$ space with its weighted version $L_p^w$. Then one gets the estimate \eqref{weightedd}, but with a worse dependence on the characteristic: $[w]_{A_p}^{3\max\{1/(p-1),1\}}$.
\end{remark}

Now let us fix some additional notation. From now on, we will work with a single dyadic lattice $\mathbb{D}_{1,0}$. Given $\Omega\in \mathbb{D}_{1,0}$ with $|\Omega|=4^N$ for some integer $N$, we introduce its filtration $(\F_n^\Omega)_{n\geq 0}$ defined by $\F_0^\Omega=\{\emptyset,\Omega\}$ and, for any $n\geq 0$,
\begin{align*}
\F_{2n+1}^\Omega&=\sigma\Big(\big\{\varphi_I\,:\,I\mbox{ is a dyadic subinterval of }\Omega,\,|I|=4^{-n}|\Omega|\big\}\Big),\\
\F_{2n+2}^\Omega&=\sigma\Big(\big\{\psi_I\,:\,I\mbox{ is a dyadic subinterval of }\Omega,\,|I|=4^{-n}|\Omega|\big\}\Big).
\end{align*}
Next, suppose that $f\in L_1(\R;\mathbb{B})$ is a given function, let $\gamma=\{\gamma(I)\}_{I\in \mathbb{D}_{1,0}}$ be an arbitrary sequence bounded by $1$ and define $g^{\Omega}=\sum_{I\in \mathbb{D}_{1,0},\,I\subseteq \Omega,\,I\text{ even}}\gamma(I)\langle f,\varphi_I\rangle \psi_I$.  Let $(f_n^\Omega)_{n\geq 0}$, $(g_n^\Omega)_{n\geq 0}$ be the martingales generated by $f|_\Omega$ and $g^\Omega|_\Omega$, relative to the filtration $\F^\Omega$. It is easy to check that the associated differences are $df_0^\Omega=\frac{1}{|\Omega|}\int_\Omega f$, $dg_0^\Omega=0$ and for $n\geq 0$,
$$\begin{array}{lll}
 \displaystyle &\displaystyle df_{2n+1}^\Omega=\sum_{|I|/|\Omega|=4^{-n}} \langle f,\varphi_I\rangle \varphi_I,\qquad  &\displaystyle dg_{2n+1}^\Omega=0\\
&\displaystyle df_{2n+2}^\Omega=\sum_{|I|/|\Omega|=4^{-n}}\Big(\langle f,\zeta_I^{\text{inn}}\rangle \zeta_I^{\text{inn}}+\langle f,\zeta_I^{\text{out}}\rangle \zeta^{\text{out}}_I\Big), \qquad &\displaystyle dg_{2n+2}^\Omega=\sum_{|I|/|\Omega|=4^{-n}} \gamma(I)\langle f,\varphi_I\rangle \psi_I,
 \end{array}$$
 where $\zeta^{\text{inn}}$, $\zeta^{\text{out}}$ have been defined in Step 1 of the proof of the previous theorem.
Observe that  $\|dg_{2n+2}^\Omega\|_\mathbb{B}\leq 7\|df_{2n+1}^\Omega\|_\mathbb{B}$ for all $n$. Furthermore, note that the real-valued variables $(\|df_n^\Omega\|_\mathbb{B})_{n\geq 0}$ are predictable: for any $n\geq 1$, $\|df_n^\Omega\|_\mathbb{B}$ is $\F_{n-1}^\Omega$-measurable.

\begin{theorem}\label{UMD_theorem3}
Under the above notation, there is a universal constant $C$ for which
\begin{equation}\label{auxil_weak}
 \left\|\sup_{n\geq 0}\Big\|g_n^\Omega\Big\|_\mathbb{B}\right\|_{L_{1,\infty}(\Omega;\R)} \leq C\left\|f\right\|_{L_{1}(\Omega;\mathbb{B})}.
\end{equation}
\end{theorem}
\begin{proof}
 We will use the previous theorem combined with the extrapolation (good-lambda) method of Burkholder and Gundy.
Fix $\beta>1$, $\delta \in (0,1)$ (the values will be specified later) and introduce the stopping times $\mu,\,\nu,\,\sigma$ by
\begin{align*}
 \mu&=\inf\{n\geq 0: \|g_n^\Omega\|_\mathbb{B}\geq 1 \},\\
\nu& =\inf\{n\geq 0:\|g_n^\Omega\|_\mathbb{B}\geq \beta\}\\
\sigma&=\inf\left\{n\geq 0: \|f_n^\Omega\|_\mathbb{B} \vee \|df_{n+1}^\Omega\|_\mathbb{B}\geq \delta \right\},
\end{align*}
with the standard convention $\inf\emptyset=\infty$ and $a\vee b=\max\{a,b\}$. To see that $\sigma$ is also a stopping time, one needs to refer to the predictability of $(\|df_n^\Omega\|_\mathbb{B})_{n\geq 0}$ discussed above.  Denoting by $a\wedge b$ the minimum of $a$ and $b$, we may write
\begin{equation}\label{Chebyshev}
\begin{split}
 \mathbb{P}\left(\sup_{n\geq 0}\|g_n^\Omega\|_\mathbb{B}\geq \beta,\,\sup_{n\geq 0}(\|f_n^\Omega\|_\mathbb{B}\vee\|df_{n+1}^\Omega\|_\mathbb{B})<\delta \right)&=\mathbb{P}(\mu\leq \nu <\infty,\,\sigma=\infty)\\
 &\leq \mathbb{P}(\|g_{\nu\wedge \sigma}^\Omega-g_{\mu\wedge \sigma}^\Omega\|_\mathbb{B}\geq \beta-1-7\delta).
\end{split}
\end{equation}
Here the latter passage is due to the triangle inequality: on the set $\{\mu\leq \nu <\infty,\,\sigma=\infty\}$ we have $\|g_{\nu\wedge \sigma}^\Omega\|_\mathbb{B}\geq \beta$  and $ \|g_{\mu\wedge \sigma}^\Omega\|_\mathbb{B}\leq \|g_{\mu\wedge \sigma-1}^\Omega\|_\mathbb{B}+7\|df_{\mu\wedge \sigma-1}^\Omega\|_\mathbb{B}\leq 1+7\delta$.
 Now, by Chebyshev's inequality, the last expression in \eqref{Chebyshev} does not exceed $\|g_{\nu\wedge \sigma}^\Omega-g_{\mu\wedge \sigma}^\Omega\|_{L_2(\Omega;\mathbb{B})}^2/(\beta-1-7\delta)^2$. The previous theorem implies that
$$\|g_{\nu\wedge \sigma}^\Omega-g_{\mu\wedge \sigma}^\Omega\|_{L_2(\Omega;\mathbb{B})}\leq C_2\|f_{\nu\wedge\sigma}^\Omega-f_{\mu\wedge \sigma}^\Omega\|_{L_2(\Omega;\mathbb{B})}.$$
(Indeed, set $f:=f_{\nu\wedge\sigma}^\Omega-f_{\mu\wedge \sigma}^\Omega$ and use the same transforming sequence $(\gamma(I))_{I\in \mathbb{D}_{1,0}}$).
Hence we obtain
\begin{align*}
 \mathbb{P}\left(\sup_{n\geq 0}\|g_n^\Omega\|_\mathbb{B}\geq \beta,\,\sup_{n\geq 0}(\|f_n^\Omega\|_\mathbb{B}\vee\|df_{n+1}^\Omega\|_\mathbb{B})<\delta \right)&\leq \frac{C_2^2\E \|f_{\nu\wedge\sigma}^\Omega-f_{\mu\wedge \sigma}^\Omega\|_\mathbb{B}^2}{(\beta-1-7\delta)^2}\\
&=\frac{C_2^2\E \|f_{\nu\wedge\sigma}^\Omega-f_{\mu\wedge \sigma}^\Omega\|_\mathbb{B}^2\chi_{\{\mu<\infty\}}}{(\beta-1-7\delta)^2},
\end{align*}
where the latter passage is due to the identity $f_{\nu\wedge\sigma}^\Omega=f_{\mu\wedge \sigma}^\Omega$ on the set $\mu=\infty$. But by the definition of $\sigma$, we have $ \|f_{\nu\wedge\sigma}^\Omega-f_{\mu\wedge \sigma}^\Omega\|_\mathbb{B}\leq \|f_{\nu\wedge\sigma}^\Omega\|_\mathbb{B}+\|f_{\mu\wedge \sigma}^\Omega\|_\mathbb{B}\leq 4\delta.$ Now, since  $\mathbb{P}(\mu<\infty)=\mathbb{P}\left(\sup_{n\geq 0}\|g_n^\Omega\|_\mathbb{B}\geq 1\right)$, putting all the above observations together gives
$$ \mathbb{P}\left(\sup_{n\geq 0}\|g_n^\Omega\|_\mathbb{B}\geq \beta,\,\sup_{n\geq 0}(\|f_n^\Omega\|_\mathbb{B}\vee\|df_{n+1}^\Omega\|_\mathbb{B})<\delta \right)\leq \frac{16C_2^2\delta^2}{(\beta-1-7\delta)^2}\mathbb{P}\left(\sup_{n\geq 0}\|g_n^\Omega\|_\mathbb{B}\geq 1\right).$$
Now we specify $\beta=3$ and $\delta=(32C_2)^{-1}$, and apply homogeneity argument to obtain that
$$ \mathbb{P}\left(\sup_{n\geq 0}\|g_n^\Omega\|_\mathbb{B}\geq 3\lambda ,\,\sup_{n\geq 0}(\|f_n^\Omega\|_\mathbb{B}\vee\|df_{n+1}^\Omega\|_\mathbb{B})<\delta \lambda\right)\leq \frac{1}{12}\mathbb{P}\left(\sup_{n\geq 0}\|g_n^\Omega\|_\mathbb{B}\geq \lambda\right)$$
 for $\lambda>0$ (here we used the fact that $\delta<1/4$, so $\beta-1-7\delta\geq \sqrt{3}$). This implies
 $$ \mathbb{P}\left(\sup_{n\geq 0}\|g_n^\Omega\|_\mathbb{B}\geq 3\lambda\right)\leq \mathbb{P}\left(\sup_{n\geq 0}(\|f_n^\Omega\|_\mathbb{B}\vee\|df_{n+1}^\Omega\|_\mathbb{B})\geq \delta \lambda\right)+\frac{1}{12}\mathbb{P}\left(\sup_{n\geq 0}\|g_n^\Omega\|_\mathbb{B}\geq \lambda\right)$$
and hence, multiplying both sides by $\lambda$, we obtain
\begin{align*}
 &\frac{1}{3}\left\|\sup_{n\geq 0}\|g_n^\Omega\|_\mathbb{B}\right\|_{L_{1,\infty}(\Omega;\R)}\leq 32C_2\left\|\sup_{n\geq 0}(\|f_n^\Omega\|_\mathbb{B}\vee \|df_{n+1}^\Omega\|_\mathbb{B})\right\|_{L_{1,\infty}(\Omega;\R)}+\frac{1}{12}\left\|\sup_{n\geq 0}\|g_n^\Omega\|_\mathbb{B}\right\|_{L_{1,\infty}(\Omega;\R)}.
\end{align*}
It remains to observe that by the triangle inequality and the weak-type $(1,1)$ bound for the (sub-)martingale maximal function,
\begin{align*}
\left\|\sup_{n\geq 0}(\|f_n^\Omega\|_\mathbb{B}\vee \|df_{n+1}^\Omega\|_\mathbb{B})\right\|_{L_{1,\infty}(\Omega;\R)}&\leq
\left\|\sup_{n\geq 0}\|f_n\|_\mathbb{B}\right\|_{L_{1,\infty}(\Omega;\R)}+\left\|\sup_{n\geq 0}\|df_{n+1}^\Omega\|_\mathbb{B}\right\|_{L_{1,\infty}(\Omega;\R)}\\
&\leq 5\left\|\sup_{n\geq 0}\|f_n\|_\mathbb{B}\right\|_{L_{1,\infty}(\Omega;\R)}\leq 5\|f\|_{L_1(\Omega;\mathbb{B})}.
\end{align*}
The proof is complete.
\end{proof}

We turn our attention to the sparse domination. Let $\mathscr{D}$ denote the class of all dyadic subintervals of $[0,1)$ having measure $4^{-n}$ for some $n$.

\begin{definition}
A collection $\mathscr{S}\subset \mathscr{D}$ is called sparse, if there is a family $\{E(\Omega)\}_{\Omega\in \mathscr{S}}$ of pairwise disjoint sets such that $E(\Omega)\subseteq \Omega$ and $|E(\Omega)|\geq |\Omega|/2$ for all $\Omega\in \mathscr{S}$.
\end{definition}

\begin{proposition}
Let $f:\R\to \mathbb{B}$ be a Bochner integrable function and let $\gamma=\{\gamma(I)\}_{I\in \mathscr{D}}$ be a sequence with values in $[-1,1]$. Then there exists a sparse family $\mathscr{S}\subset\mathscr{D}$ for which we have
\begin{equation}\label{sparse_domination}
 \left\|\sum_{I\in \mathscr{D}} \gamma(I)\langle f,\varphi_I\rangle \psi_I\right\|_\mathbb{B}\leq (2C+7)\sum_{\Omega\in \mathscr{S}} \left(\frac{1}{|\Omega|}\int_\Omega \|f\|_\mathbb{B}\right)\chi_\Omega
\end{equation}
almost everywhere on $[0,1)$. Here $C$ is the weak-type constant in \eqref{auxil_weak}.
\end{proposition}

\begin{proof}The collection $\mathscr{S}$ will be obtained by the following algorithm.

\smallskip

\emph{Step 1.} We put $[0,1)$ into $\mathscr{S}$ and mark it as `unused'.

\smallskip

\emph{Step 2.} We pick an unused element $\Omega \in \mathscr{S}$ of maximal measure and define  $\lambda_\Omega=\frac{2C}{|\Omega|}\int_\Omega \|f\|_\mathbb{B}$. Consider the martingale $(g^\Omega_n)_{n\geq 0}$ and split the set $ \{\omega\in \Omega:\sup_{n\geq 0}\|g^\Omega_n\|_\mathbb{B}\geq \lambda_\Omega\}$ into the union of pairwise disjoint and maximal elements $\Omega_1$, $\Omega_2$, $\ldots$ of $\mathscr{D}$. There might be finite or infinite number of such terms, we put them all into $\mathscr{S}$.

\smallskip

\emph{Step 3.} We define $ E(\Omega)=\{\omega\in \Omega:\sup_{n\geq 0}\|g^\Omega_n\|_\mathbb{B}<\lambda_\Omega\}$, mark $\Omega$ as `used' and go to Step 2.

\smallskip

Let us study the properties of the above objects. The class $\mathscr{S}$ we obtain is indeed contained in $\mathscr{D}$. By the construction, the sets $\{E(\Omega)\}_{\Omega\in \mathscr{S}}$ are pairwise disjoint, furthermore, the weak-type inequality \eqref{auxil_weak} implies $|E(\Omega)|\geq |\Omega|/2$ for any $\Omega\in \mathscr{S}$. This in particular gives $\sum_{\Omega \in \mathscr{S}}|\Omega|\leq 2$ and hence almost all $\omega\in [0,1)$ belong to a finite number of elements of $\mathscr{S}$. Let $j(w)$ be the unique positive integer such that $\omega\in E(\Omega_{j(\omega)}^\omega)\subset \Omega_{j(\omega)}^\omega\subset \Omega_{j(\omega)-1}^\omega\subset \ldots \subset \Omega_1^\omega=[0,1)$, with $\Omega_j^\omega\in \mathscr{S}$.

We are ready to verify \eqref{sparse_domination}. Outside $[0,1)$ both sides vanish, and for $\omega\in [0,1)$ we write
\begin{align*}
&\left\|\sum_{I\in \mathscr{D}} \gamma(I)\langle f,\varphi_I\rangle \psi_I(\omega)\right\|_\mathbb{B}\\
&\leq
\sum_{k=1}^{j(\omega)}\left\|\sum_{I\in \mathscr{D},\Omega_{k-1}^\omega\supseteq I\supsetneq \Omega_k^\omega} \gamma(I)\langle f,\varphi_I\rangle \psi_I(\omega)\right\|_\mathbb{B}+\left\|\sum_{I\in \mathscr{D},\Omega_{j(\omega)}^\omega\supseteq I} \gamma(I)\langle f,\varphi_I\rangle \psi_I(\omega)\right\|_\mathbb{B}.
\end{align*}
However, for any $\Omega$, the partial sums of $\sum_{I\in \mathscr{D},\Omega\supseteq I} \gamma(I)\langle f,\varphi_I\rangle \psi_I$ form the martingale $g^\Omega$. Thus, by the very definition of the splitting procedure in Step 2, we have
$$  \left\|\sum_{I\in \mathscr{D},\Omega_{j(\omega)}^\omega\supseteq I} \gamma(I)\langle f,\varphi_I\rangle \psi_I(\omega)\right\|_\mathbb{B}\leq \frac{2C}{|\Omega_{j(\omega)}^\omega|}\left(\int_{\Omega_{j(\omega)}^\omega} \|f\|_\mathbb{B}\right) \chi_{\Omega_{j(\omega)}^\omega}(\omega).$$
For the expression
$$ \left\|\sum_{I\in \mathscr{D},\Omega_{k-1}^\omega\supseteq I\supsetneq \Omega_k^\omega} \gamma(I)\langle f,\varphi_I\rangle \psi_I(\omega)\right\|_\mathbb{B}$$
we proceed similarly, however, we need a small modification, as the above construction shows that this is \emph{larger} than $\frac{2C}{|\Omega_{k-1}^\omega|}\int_{\Omega_{k-1}^\omega}\|f\|_\mathbb{B}$. Denoting the parent of $\Omega_k^\omega$ in $\mathscr{D}$ by $(\Omega_k^\omega)'$, we obtain
\begin{align*}
& \left\|\sum_{I\in \mathscr{D},\Omega_{k-1}^\omega\supseteq I\supsetneq \Omega_k^\omega} \gamma(I)\langle f,\varphi_I\rangle \psi_I(\omega)\right\|_\mathbb{B}\\
&\leq  \left\|\sum_{I\in \mathscr{D},\Omega_{k-1}^\omega\supseteq I\supsetneq (\Omega_k^\omega)'} \gamma(I)\langle f,\varphi_I\rangle \psi_I(\omega)\right\|_\mathbb{B}
+7\|\gamma\|_\infty \|\langle f,\varphi_{(\Omega_k^\omega)'}\rangle\|_\mathbb{B}|(\Omega_k^\omega)'|^{-1/2}\\
&\leq 2C\left(\frac{1}{|\Omega_{k-1}^\omega|}\int_{\Omega_{k-1}^\omega}\|f\|_\mathbb{B}\right)\chi_{\Omega_{k-1}^\omega}(\omega)+7\left(\frac{1}{|\Omega_k^\omega|}\int_{\Omega_k^\omega}\|f\|_\mathbb{B}\right)\chi_{\Omega_k^\omega}(\omega).
\end{align*}
This gives the claim.
\end{proof}

Finally, we are ready for the proof of the weighted estimate \eqref{weightedd}. The change-of-measure argument used below is inspired by \cite{Mo}.

\begin{proof}[Proof of Theorem \ref{UMD_theorem}] We start with reductions. It suffices to show the claim for $p\geq 2$, then the case $1<p<2$ follows by duality. Next, by the approximation, scaling and translating, it is enough to show that
$$  \left\|\sum_{I\in \mathscr{D}} \gamma(|I|)\langle f,\varphi_I\rangle \psi_I\right\|_{L_p^w(\R;\mathbb{B})}\leq C_p\|\gamma\|_\infty [w]_{A_p}^{\max\{1/(p-1),1\}}\|f\|_{L_p^w(\R;\mathbb{B})}.
$$
By homogeneity, we may and do assume that $\|\gamma\|_\infty\leq 1$.  Therefore, using the \eqref{sparse_domination}, we will be done if we prove the estimate
$$ \left\|\sum_{\Omega\in \mathscr{S}} \left(\frac{1}{|\Omega|}\int_\Omega \|f\|_\mathbb{B}\right)\chi_\Omega\right\|_{L_p^w(\R;\mathbb{B})}\leq C_p\|\gamma\|_\infty [w]_{A_p}^{\max\{1/(p-1),1\}}\|f\|_{L_p^w(\R;\mathbb{B})}.$$
To this end, we let $v=w^{1/(1-p)}$ be the dual weight to $w$ and pick an arbitrary nonnegative $h\in L_{p'}^v(\R;\R)$. For any $I\in \mathscr{D}$ and any weight $u$, the symbol $\mathcal{E}_{I}^uf=\frac{1}{u(I)}\int_I fu\mbox{d}\omega$ will stand for the average of $f$ over $I$ with respect to the measure $ud\omega$. Then
\begin{align*}
&\int_0^1 \left(\sum_{\Omega\in \mathscr{S}} \left(\frac{1}{|\Omega|}\int_\Omega \|f\|_\mathbb{B}\right)\chi_\Omega\right)h \mbox{d}\omega\\
&=\sum_{\Omega \in \mathscr{S}} \frac{w(\Omega)v(\Omega)^{p-1}}{|\Omega|^p} \cdot |\Omega|^{p-1}v(\Omega)^{2-p}\mathcal{E}_\Omega^v (\|f\|_\mathbb{B}v^{-1})\mathcal{E}_\Omega^w(hw^{-1})\\
&\leq [w]_{A_p} \sum_{\Omega \in \mathscr{S}} |\Omega|^{p-1}v(\Omega)^{2-p}\mathcal{E}_\Omega^v (\|f\|_\mathbb{B}v^{-1})\mathcal{E}_\Omega^w(hw^{-1})\\
&\leq 2^{p-1}[w]_{A_p}\sum_{\Omega \in \mathscr{S}} |E(\Omega)|^{p-1}v(\Omega)^{2-p}\mathcal{E}_\Omega^v (\|f\|_\mathbb{B}v^{-1})\mathcal{E}_\Omega^w(hw^{-1}),
\end{align*}
where in the last passage we have used the sparseness estimate $|\Omega|\leq 2|E(\Omega)|$ for $\Omega\in\mathscr{S}$. Since $p\geq 2$ and $E(\Omega)\subset \Omega$, we have $v(\Omega)^{2-p}\leq v(E(\Omega))^{2-p}$. % and hence
%\begin{align*}
%& \int_{\R^n} T^\mathscr{S} (f\sigma)gw\mbox{d}x\\
%&\leq 2^{p-1}[w]_{A_p}\sum_{\Omega\in \mathscr{S}}\frac{1}{\sigma(\Omega)}\int_\Omega f\sigma\mbox{d}x\cdot \frac{1}{w(\Omega)}\int_\Omega %gw\mbox{d}x \cdot |E(\Omega)|^{p-1}\sigma(E(\Omega))^{2-p}.
%\end{align*}
Furthermore, by H\"older's inequality, we see that
$ |E(\Omega)|\leq w(E(\Omega))^\frac{1}{p}v(E(\Omega))^\frac{1}{p'},$
so $ |E(\Omega)|^{p-1}v(E(\Omega))^{2-p}\leq v(E(\Omega))^\frac{1}{p}w(E(\Omega))^\frac{1}{p'}.$
Plugging these observations above and applying H\"older's inequality again, we get
\begin{align*}
&\sum_{\Omega \in \mathscr{S}} |E(\Omega)|^{p-1}v(E(\Omega))^{2-p}\mathcal{E}_\Omega^v (\|f\|_\mathbb{B}v^{-1})\mathcal{E}_\Omega^w(hw^{-1})\\
&\leq \sum_{\Omega\in \mathscr{S}}v(E(\Omega))^\frac{1}{p}w(E(\Omega))^\frac{1}{p'}\cdot  \mathcal{E}_\Omega^v (\|f\|_\mathbb{B}v^{-1})\mathcal{E}_\Omega^w(hw^{-1}) \\
&\leq \left(\sum_{\Omega\in \mathscr{S}} \left(\mathcal{E}_\Omega^v (\|f\|_\mathbb{B}v^{-1})\right)^pv(E(\Omega))\right)^\frac{1}{p}\left(\sum_{\Omega\in \mathscr{S}} \left(\mathcal{E}_\Omega^w(hw^{-1})\right)^{p'}w(E(\Omega))\right)^\frac{1}{p'}\\
&\leq \|M_v (\|f\|_{\mathbb{B}}v^{-1})\|_{L_p^v(\R;\R)}\|M_w (hw^{-1})\|_{L_{p'}^w(\R;\R)}\\
&\leq pp'\big\|\|f\|_{\mathbb{B}}v^{-1}\big\|_{L_p^v(\R;\R)}\big\|hw^{-1}\big\|_{L_{p'}^w(\R;\R)}=pp'\|f\|_{L_p^w(\R;\mathbb{B})}\|h\|_{L_{p'}^v(\R;\R)}.
\end{align*}
Here $M_w$ and $M_v$ are the classical dyadic maximal operators with respect to the measures $w$ and $v$, respectively. This yields the desired assertion by taking the supremum over all $h$ as above.
\end{proof}

\appendix
\section{Alternative approaches to Theorem \ref{ref}}

Now we will discuss different approaches, which unfortunately do not seem to yield the sharp dependence on the $A_p$ characteristic for any $p$. Anyhow, we believe that the alternative argumentation is of its own interest and connections, which might be useful in other contexts. We would also like to point out that the first method below (which rests on interpolation) does give the optimal exponent in the classical case.

\subsection{Marcinkiewicz-type interpolation}
In the following, we use the weighted weak type bound established in Theorem \ref{weak_theorem} and the  Marcinkiewicz-type interpolation to prove
the weighted maximal $L_p$ inequalities with a suboptimal exponent.

\begin{proof}[Proof of Theorem \ref{ref} with a suboptimal exponent $2/(p-1)$] Let us first recall an interpolation result established by Dirksen \cite{Dir} (see also Junge and Xu \cite{jungexu}). Let $1\leq r<p<q$ be two parameters and suppose that $\left(T_\alpha\right)_{\alpha \in A}$ is  a net of positive, subadditive maps on $L_0(\mathcal{M})$, which is of weak type $\left(r,r\right)$ with a constant $C_r$ and of weak type $(q,q)$ with a constant $C_q$. Then we have

\begin{displaymath}
\left\|\left(T_\alpha\left(x\right)\right)_{\alpha \in A}\right\|_{L_p\left(\mathcal{M}; \ell_\infty\right)} \lesssim \max\{C_r,C_q\}\left(\frac{pr}{p-r}+\frac{pq}{q-p}\right)^2\|x\|_{L_p\left(\mathcal M\right)}.
\end{displaymath}
Also, recall the celebrated self-improvement property of dyadic $A_p$ weights established by Coifman and Fefferman \cite{CF}: if $w \in A_p$ for some $p>1$, then $w \in A_{p-\varepsilon}$ for $\varepsilon \approx [w]_{A_p}^{-\frac{1}{p-1}}$ and $[w]_{A_{p-\varepsilon}} \lesssim[w]_{A_p}$.
Now we will combine all the above facts to obtain the desired estimate. Let $1<p<\infty$ and suppose that $ {  w}$ is an $A_p$ weight. We use interpolation for maps $\left(\mathcal{E}_n\right)_{n \geq 0}$ with $r = p - \varepsilon$ and $q = p+\varepsilon$, where $\e$ is as above. Then $ {  w}\in A_{p-\e}$ and $ {  w}\in A_{p+\e}$ (since $A_p\subset A_{p+\e}$), so the weak-type estimates hold true due to Theorem \ref{weak_theorem}. This yields the desired weighted Doob's maximal inequality with the constant $c_p[ {  w}]_{A_p}^{2/(p-1)}$. By duality, this implies the estimate
\begin{equation}\label{quadratic}
 \left\|\sum_{n\geq 0}\mathcal{E}_n(a_n)\right\|_{L_p^ {  w}(\mathcal{M})}\lesssim [ {  w}]_{A_p}^2\left\|\sum_{n\geq 0} a_n\right\|_{L_p^ {  w}(\mathcal{M})},\qquad 1<p<\infty,
\end{equation}
with the suboptimal, quadratic dependence on the characteristic.
\end{proof}

\subsection{Factorization and complex interpolation} There is a natural question whether the above exponent $2/(p-1)$ can be improved with the use of structural properties of Muckenhoupt's weights. This question is motivated by the trivial bound
\begin{equation}\label{trivial}
 \left\|\sum_{n\geq 0}\mathcal{E}_n(a_n)\right\|_{L_1^ {  w}(\mathcal{M})}\leq  [ {  w}]_{A_1}\left\|\sum_{n\geq 0} a_n\right\|_{L_1^ {  w}(\mathcal{M})},
\end{equation}
which gives hope that some interpolation arguments might lead to an improvement.
We start with the factorization of weights. In the statement below, we work on the (classical) measure space $(X,\F,\mu)$ equipped with some filtration $(\F_n)_{n\geq 0}$.

\begin{theorem}
Let $1<p<\infty$ and suppose that $w$ is an $A_p$ weight. Then there exist $A_1$ weights $w_1$, $w_2$ satisfying $[w_1]_{A_1}\lesssim [w]_{A_p}$, $[w_2]_{A_1}\lesssim [w]_{A_p}^{1/(p-1)}$ and $w=w_1w_2^{1-p}.$
\end{theorem}
\begin{proof}
Suppose first that $p\geq 2$. Let $M$ be the classical maximal operator associated with $(\F_n)_{n\geq 0}$. Introduce the auxiliary operator $T$ acting on nonnegative random variables by
$$ T(f)=(w^{-1/p}M(f^{p-1}w^{1/p}))^{1/(p-1)}+w^{1/p}M(fw^{-1/p}).$$
This operator is well-defined and bounded on (unweighted) $L_p(X,\F,\mu)$: this follows at once from the fact that $M$ maps $L_{p'}(w^{-1/(p-1)})$ to itself and $L_p(w)$ to itself. Actually, since $\|M\|_{L_p^w\to L_p^w}\lesssim [w]_{A_p}^{1/(p-1)}$ and $\|M\|_{L_{p'}^{w^{-1/(p-1)}}\to L_{p'}^{w^{-1/(p-1)}}}\lesssim  [w^{-1/(p-1)}]_{A_{p'}}^{1/(p'-1)}=[w]_{A_p}$, we obtain
\begin{equation}\label{mind}
 \|T\|_{L_p\to L_p}\lesssim [w]_{A_p}^{1/(p-1)}.
\end{equation}
Furthermore,  the operator $T$ is positive and sublinear (the latter property holds since $p\geq 2$). Now, define $\varphi \in L_p$ as the sum of $L_p$ convergent series
$$ \varphi=\sum_{n=1}^\infty (2\|T\|_{L_p\to L_p})^{-n}T^n(\psi),$$
where $\psi$ is an arbitrary fixed  norm-one element of $L_p(X)$.
We define the $A_1$ factors by $w_1=w^{1/p}\varphi^{p-1}$ and $w_2=w^{-1/p}\varphi$, so that $w=w_1w_2^{1-p}$. Directly by the definition of $\varphi$, we get
$$ T\varphi\leq 2\|T\|_{L_p\to L_p}\sum_{n=1}^\infty (2\|T\|_{L_p\to L_p})^{-n-1}T^{n+1}( \psi)\leq 2\|T\|_{L_p\to L_p}\varphi,$$
which is equivalent to
$$ (w^{-1/p}M(\varphi^{p-1}w^{1/p}))^{1/(p-1)}+w^{1/p}M(\varphi w^{-1/p})\leq 2\|T\|_{L_p\to L_p}\varphi.$$
But $\varphi=(w^{-1/p}w_1)^{1/(p-1)}$, so we obtain $M(w_1)\leq \big(2\|T\|_{L_p\to L_p}\big)^{p-1}w_1$ and $M(w_2)\leq 2\|T\|_{L_p\to L_p}w_2$.
It remains to apply \eqref{mind} to complete the analysis for $p\geq 2$.

If $1<p<2$, we pass to the dual weight $v=w^{-1/(p-1)}\in A_{p'}$. By what we have just proved, $v=v_1v_2^{1-p'}$ with $[v_1]_{A_1}\lesssim  [v]_{A_{p'}}=[w]_{A_p}^{1/(p-1)}$ and $[v_2]_{A_1}\lesssim [v]_{A_{p'}}^{1/(p'-1)}=[w]_{A_p}$. Thus, $w=v_2v_1^{1-p}$ is the desired factorization.
\end{proof}

We will prove the following fact.

\begin{theorem}
Suppose that for some $q>1$ and some $\kappa>0$ we have
\begin{equation}\label{assumpt}
 \left\|\sum_{n\geq 0}\mathcal{E}_n(a_n)\right\|_{L_q^ {w}(\mathcal{M})}\leq  [ {w}]_{A_q}^{\kappa}\left\|\sum_{n\geq 0} a_n\right\|_{L_q^ {w}(\mathcal{M})}
\end{equation}
for all $A_q$ weights $w$.
Then for any $1<p<q$ we have
$$  \left\|\sum_{n\geq 0}\mathcal{E}_n(a_n)\right\|_{L_p^ {w}(\mathcal{M})}\leq  [ {w}]_{A_p}^\gamma\left\|\sum_{n\geq 0} a_n\right\|_{L_p^ {w}(\mathcal{M})},$$
where $\gamma=1-q'/p'+\kappa(q'/p'+q/p)$.
\end{theorem}
\begin{proof}
Let $\theta\in (0,1)$ be uniquely determined by the condition $p^{-1}=(1-\theta)+\theta q^{-1}$ (that is, $\theta=q'/p'$). Assume that $w$ is an arbitrary $A_p$ weight and let $w=w_1w_2^{1-p}$ be the factorization granted by the previous theorem; let us also distinguish the weight $w_0=w_1w_2$. Suppose that $(x_n)_{n\geq 0}$ is a finite sequence in $L_{2p}(\mathcal{M})$, satisfying $\left\|\left(\sum_{n\geq 0}|x_n|^2\right)^{1/2}\right\|_{L_{2p}^{ {  w}_0}(\mathcal{M})}\leq 1$ and let $b$ be an element of $L_{p'}^{ {  w}_0}(\mathcal{M})$ of norm not exceeding one. By the results of Kosaki \cite{Ko}, there exist continuous functions $X_n,\,B:\{z\in \mathbb{C}:0\leq \Re z\leq 1\}\to \mathcal{M}$, $n=0,\,1,\,2,\,\ldots$, analytic in the interior of the strip, such that $X_n(\theta)=x_n$, $B(\theta)=b$,
$$\max\left\{\left\|\left(\sum_{n\geq 0}|X_n(it)|^2\right)^{1/2}\right\|_{L_{2}^{ {  w}_0}(\mathcal{M})},\left\|\left(\sum_{n\geq 0}|X_n(1+it)|^2\right)^{1/2}\right\|_{L_{2q}^{ {  w}_0}(\mathcal{M})}\right\}\leq 1$$
and
$$ \max\left\{\|B(it)\|_{L_{\infty}^{ {  w}_0}(\mathcal{M})},\|B(1+it)\|_{L_{q'}^{ {  w}_0}(\mathcal{M})}\right\}\leq 1.$$
Consider the analytic function
$$ F(z)=\tau \left(\sum_{n\geq 0} \mathcal{E}_n\big(|X_n(z)|^2 {  w}_2\big)B(z)  {  w}_1\right)$$
defined for $z\in \mathbb{C}$ with $0\leq \Re z\leq 1$. We have
$$ |F(it)|\leq \tau\left(\sum_{n\geq 0}\mathcal{E}_n\big(|X_n(it)|^2 {  w}_2\big) {  w}_1\right)\|B(it)\|_{L_{\infty}^{ {  w}_0}(\mathcal{M})}\leq \tau\left(\sum_{n\geq 0}\mathcal{E}_n\big(|X_n(it)|^2 {  w}_2\big) {  w}_1\right).$$
Putting $a_n=|X_n(it)|^2 {  w}_2$, we see that \eqref{trivial} yields
\begin{align*}
|F(it)|\leq [w_1]_{A_1}\tau\left(\sum_{n\geq 0}a_n w_1\right)=[w_1]_{A_1}\left\|\left(\sum_{n\geq 0}|X_n(it)|^2\right)^{1/2}\right\|_{L_{2}^{ {  w}_0}(\mathcal{M})}^2\leq [w_1]_{A_1}.
\end{align*}
Furthermore, by H\"older's inequality,
\begin{align*}
|F(1+it)|&\leq \left\|\sum_{n\geq 0}\mathcal{E}_n\big(|X_n(1+it)|^2 {  w}_2\big)\right\|_{L_q^{ {  w}_1 {  w}_2^{1-q}}(\mathcal{M})}\|B(1+it)\|_{L_{q'}^{ {  w}_0}(\mathcal{M})}\\
&\leq \left\|\sum_{n\geq 0}\mathcal{E}_n\big(|X_n(1+it)|^2 {  w}_2\big)\right\|_{L_q^{ {  w}_1 {  w}_2^{1-q}}(\mathcal{M})}.
\end{align*}
However, by a simple application of H\"older's inequality, we get $[  {  w}_1 {  w}_2^{1-q}]_{A_q}\leq [ {  w}_1]_{A_1}[ {  w}_2]_{A_1}^{q-1}$ and hence \eqref{assumpt}, applied to $a_n=|X_n(1+it)|^2 {  w}_2$, gives
$$|F(1+it)|\lesssim \left([ {  w}_1]_{A_1}[ {  w}_2]_{A_1}^{q-1}\right)^{\kappa}\left\|\left(\sum_{n\geq 0}|X_n(1+it)|^2\right)^{1/2}\right\|_{L_{2q}^{ {  w}_0}(\mathcal{M})}^2\leq \left([ {  w}_1]_{A_1}[ {  w}_2]_{A_1}^{q-1}\right)^\kappa.$$
Consequently, by the three lines lemma, we get $|F(\theta)|\lesssim [w_1]_{A_1}^{1-\theta}\left([ {  w}_1]_{A_1}[ {  w}_2]_{A_1}^{q-1}\right)^{\kappa \theta}$. But $[w_1]_{A_1}\lesssim [w]_{A_p}$ and $[w_2]_{A_2}\lesssim [w]_{A_p}^{1/(p-1)}$, so
$$ \tau \left(\sum_{n\geq 0} \mathcal{E}_n\big(|x_n|^2 {  w}_2\big)b  {  w}_1\right)\lesssim [w]_{A_p}^{1-\theta+\kappa\theta+(q-1)\kappa \theta/(p-1)}.$$
Recall that $\theta=q'/p'$. Taking the supremum over all $b$ as above, we obtain
$$ \left\|\sum_{n\geq 0} \mathcal{E}_n\big(|x_n|^2 {  w}_2\big)\right\|_{L_p^{ {  w}}(\mathcal{M})}
\lesssim [w]_{A_p}^{\gamma}
\left\|\left(\sum_{n\geq 0}|x_n|^2\right)^{1/2}\right\|_{L_{2p}^{ {  w}_0}(\mathcal{M})}^2.$$
It remains to plug $a_n=|x_n|^2 {  w}_2$, $n=0,\,1,\,2,\,\ldots$ to get the claim.
\end{proof}

\end{document}